\documentclass[10pt]{amsart}
\usepackage{amssymb, amscd, amsmath, amsthm, epsf, epsfig, latexsym}
\usepackage{pb-diagram, fancyhdr, graphicx, psfrag}

\newcommand{\compactlist}{\begin{list}{$\bullet$}{\setlength{\leftmargin}{1em}}}

\def\co{\colon\thinspace}
\def\cs{\mathop{\#}}
\def\calr{\mathcal{R}}

 \def\calc{\mathcal{C}}

\newcommand{\fig}[2] { \includegraphics[scale=#1]{#2} }

\newcommand{\Cdoub}{\CFKinf(D(T_{2,3}))}
\newcommand{\doub}{D(T_{2,3})}
\newcommand{\HF}{HF}
\newtheorem{thm}{Theorem}
\def\bull{\vbox{\hrule\hbox{\vrule\kern3pt\vbox{\kern6pt}\kern3pt\vrule}\hrule}}

\newcommand{\cm}{\cdot}

\newcommand{\F}{\mathbb F}
\newcommand{\Z}{\mathbb Z}
\newcommand{\Q}{\mathbb Q}

\newcommand\relspinc{\underline{\spinc}}

\newcommand\Filt{\mathcal F}

\newcommand\x{\mathbf x}

\newcommand\y{\mathbf y}

\newcommand\ModFlow{\mathcal M}

\newcommand\ModSphere{\ModFlow\left({\mathbb S}\longrightarrow 
\Sym^{g-1}(\Sigma_{1})\times \Sym^2(\Sigma_{2})\right)}
\newcommand\ModSpheres\ModSphere
\newcommand\CF{CF}

\newcommand\CFa{\widehat{CF}}
\newcommand\CFp{\CFb}
\newcommand\CFm{\CF^-}

\newcommand\HFp{\HFb}

\newcommand\HFm{\HF^-}
\newcommand\CFinf{CF^\infty}
\newcommand\HFinf{HF^\infty}
\newcommand\CFb{CF^+}
\newcommand\HFa{\widehat{HF}}
\newcommand\HFb{HF^+}
\newcommand\gr{\mathrm{gr}}

\newcommand\UnparModSp{\widehat \ModSp}
\newcommand\UnparModFlow\UnparModSp
\newcommand\Mod\ModSp

\newcommand{\spinc}{\ifmmode{{\mathfrak s}}\else{${\mathfrak s}$\ }\fi}

\newcommand{\spinct}{\mathfrak t}

\newcommand\ModMaps{\mathcal M}
\newcommand\ModSp\ModMaps

\newcommand\Ta{{\mathbb T}_{\alpha}}

\newcommand\Tb{{\mathbb T}_{\beta}}

\newcommand\spincrel\relspinc

\newcommand\CFK{CFK}
\newcommand\HFK{HFK}

\newcommand\CFKinf{\CFK^{\infty}}

\newcommand\HFKa{\widehat\HFK}

\newcommand\SpinC{\mathrm{Spin}^c}

\newcommand\Dual{\mathcal D}
\newcommand\Duality\Dual

\newcommand\ons{Ozsv{\'a}th and Szab{\'o}}
\newcommand\os{Ozsv{\'a}th-Szab{\'o}}

\newtheorem{theorem}{Theorem}[section]

\newtheorem{lemma}[theorem]{Lemma}
\newtheorem{corollary}[theorem]{Corollary}
\newtheorem{prop}[theorem]{Proposition}
\theoremstyle{definition}
\newtheorem{definition}[theorem]{Definition}
\def\co{\colon\thinspace}

\numberwithin{equation}{section}

\begin{document}
\title{Topologically Slice Knots with Nontrivial Alexander Polynomial}
 \author{Matthew Hedden}\author{Charles Livingston}\author{Daniel Ruberman}
\thanks{This work was supported in part by the National Science Foundation under Grants 0706979,  0906258, 0707078, 1007196, and 0804760.\\ \today}
 
\address{Matthew Hedden: Department of Mathematics, Michigan State University, East Lansing, MI 48824 }
\email{mhedden@math.msu.edu}
\address{Charles Livingston: Department of Mathematics, Indiana University, Bloomington, IN 47405 }
\email{livingst@indiana.edu}
\address{Daniel Ruberman: Department of Mathematics, MS 050, Brandeis University, Waltham, MA 02454}
\email{ruberman@brandeis.edu}


 \begin{abstract} Let  $ \calc_T$ be the subgroup of the smooth knot concordance group generated by topologically slice knots and let  $\calc_\Delta$ be the subgroup generated by knots with trivial Alexander polynomial.  We prove  $\calc_T / \calc_\Delta$ is infinitely generated.  Our methods reveal a similar structure in the $3$--dimensional rational spin bordism group, and lead to the construction of links that are topologically, but not smoothly, concordant to boundary links. \end{abstract}

\maketitle
\section{Introduction}
  Donaldson's landmark theorem~\cite{d} has an immediate corollary (first observed by Akbulut and Casson, and appearing in~\cite{cochrangompf}) that there are classical knots with trivial Alexander polynomial that are not smoothly slice.  A year after Donaldson's work, Freedman~\cite{freedman:non-simply-connected,freedman-quinn} proved a $4$--dimensional topological surgery theorem for manifolds with fundamental group $\Z$, implying that a knot with trivial Alexander polynomial is in fact topologically (locally flat) slice.  Thus the natural map from the $3$--dimensional smooth knot concordance group $\calc$ to the topological (locally flat) concordance group $\calc^{top}$ is not injective.    

This paper is concerned with the kernel of the map between the groups $\calc$ and $\calc^{top}$.     Let  $ \calc_T$ be the subgroup of $\calc$ consisting of topologically  slice knots and let  $\calc_\Delta \subset \calc_T$ be the subgroup generated by knots with trivial Alexander polynomial.  Using gauge theoretic methods of Furuta~\cite{furuta:cobordism} and Fintushel-Stern~\cite{fs:pseudo}, Endo~\cite{e} proved that $\calc_\Delta$ contains an infinitely generated subgroup.  Based on Freedman's work it seemed possible that $\calc_\Delta = \calc_T$, or in other words that every topologically slice knot is smoothly concordant to a knot with trivial Alexander polynomial.  This possibility was heightened by the observation by Cochran, Friedl, and Teichner~\cite{cochran-friedl-teichner:newslice}, that  many potential counterexamples developed by Friedl and Teichner in~\cite{ft,ftcorr}  are in fact concordant  to knots with trivial Alexander polynomial.  Our main result is that, to the contrary, there is a substantial gap between these two subgroups of 
$\calc$.
 \begin{thm}\label{infgen}   $\calc_T / \calc_\Delta$ contains an infinitely generated free subgroup.
 \end{thm} 
 
To put Theorem~\ref{infgen} in the context of known results on the structure of the knot concordance group, 
we recall Ê the following decomposition, parameterized by half-integers:
$$ \calc_\Delta \subset \calc_T \subsetÊ \cap\, \calc_{i} \subset \cdots \subset \calc_{1.5} \subset \calc_1Ê \subset
\calc_AÊ \subset \calc,$$
where $\calc_A$ is the concordance group of algebraically slice knots 
(that is, the kernel of Levine's classifying homomorphism~\cite{le}) and the $\calc_i$ 
are the terms of the Cochran-Orr-Teichner filtration~\cite{cot1}.

Early work on concordance demonstrated that $\calc$ is infinitely generated~\cite{fm, le,  mi3,  tris}.   The infinite generation of $\calc_A$ follows from the  work of Casson and
Gordon~\cite{cg}, as shown by Jiang~\cite{ji}.  The infinite generation of $\calc_2 / \calc_{2.5}$ was proved in~\cite{cot2} and the infinite generation of $\calc_i / C_{i +.5}$ for all $i$ was proved in~\cite{chl} (see also~\cite{cochran-teichner} for the existence of elements of infinite order in each of the quotients).
 Conjecturally, $\cap\,
\calc_n = \calc_T$,   but little progress has been made in proving this.  As mentioned above, in~\cite{e} it is shown
that $\calc_\Delta$ is infinitely generated.  Thus the non-triviality of the quotient $\calc_T/\calc_\Delta$, established in
this paper, is  at the heart of the distinction between the smooth and topological category.
\vskip.1in
 \noindent{\bf Summary.}  The proof of Theorem~\ref{infgen} has three parts: (1)  developing an obstruction to a knot being smoothly concordant to a knot with trivial Alexander polynomial; (2) constructing knots that offer potential examples; and (3) explicitly computing the obstructions.   
 
 The first part of the proof of Theorem~\ref{infgen} occupies sections~\ref{sectionspinc} and~\ref{sectiondandobstruct}.  Section~\ref{sectionspinc} reviews Spin$^c$ structures on $3$--manifolds.  Section~\ref{sectiondandobstruct}  describes  basic properties of the Heegaard Floer {\it correction term} $d(Y,\spinc )$, which is a rational number associated to a $3$--manifold $Y$ equipped with a Spin$^c$ structure $\spinc$.  This invariant, defined   by Ozsv\'ath and Szab\'o in~\cite{os2}, has been applied  to study knot concordance in~\cite{grs, jn, mo}). Letting $\Sigma(K)$ denote the $2$--fold branched cover of a knot $K$, we use these basic properties of $d$  to prove  that a related invariant, $\bar{d}(\Sigma(K),\spinc )$, provides an obstruction to $K$ being concordant to a knot with trivial Alexander polynomial.   {\em A priori},  demonstrating that $\calc_T / \calc_\Delta$ is infinitely generated requires knowledge of our obstruction for all linear combinations of knots in a proposed basis.  Thus, a key aspect of this part of the proof is a careful analysis of possible metabolizers for $(\Z/p^2\Z)^n$, which significantly reduces the amount of Floer theoretic computation necessary.

 The second part of the proof of Theorem~\ref{infgen} occupies Section~\ref{sectionknots}.  Here a family of knots $\{K_p\}$ is constructed and the covers $\Sigma(K_p)$ are described as surgery on knots in $S^3$.  Figure~\ref{figure1} illustrates one example, the knot $K_3$.  In that figure, the knot $J_3$ is the connected sum of two untwisted Whitehead doubles of the trefoil knot: $J_3 = D(T_{2,3}) \# D(T_{2,3})$.  For other values of $p$, all of which are selected to be primes congruent to 3 modulo 4, there are $p$ half-twists between the bands of the illustrated surface and $J_p$ consists of $(3p-1)/4$ copies of the untwisted double of the trefoil.  A relatively easy argument using Freedman's result shows that $K_p$ is topologically slice. An exercise in framed link descriptions of $3$--manifolds then shows that $\Sigma(K_p)$ can be described as surgery on a knot in $S^3$.  More precisely, $\Sigma(K_p) = S^3_{p^2}(L_{p})$; that is, by $p^2$ surgery on a knot $L_p \subset S^3$, where $L_p$ is the  connected sum of the $(p-1,p)$--torus knot with $(3p -1)/2$ untwisted doubles of the trefoil.

\begin{figure}[h]
\psfrag{j3}{$J_3$}
\fig{.5}{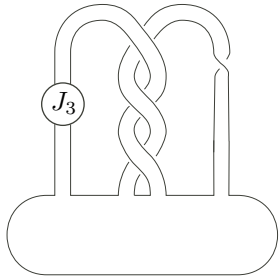}
\caption{A punctured Klein bottle whose boundary is the knot $K_3$.  The circle labeled $J_3$ indicates that we tie a knot $J_3$ into the band.    }\label{figure1}
\end{figure}
 
 The third part of the proof of Theorem~\ref{infgen} is the most technical, calling on precise estimates of the value of the invariants $d(\Sigma(K_p),\spinc )$ for appropriate Spin$^c$ structures.  In Section~\ref{sectionhf} we review the necessary background material from Heegaard Floer homology.  In Section~\ref{sectioncompute} we carry out the explicit estimates needed to obstruct the relevant knots from being concordant to knots with trivial Alexander polynomial, completing the proof of Theorem~\ref{infgen}.
 
 Section~\ref{cobordismsection}  applies the results of Section~\ref{sectionspinc} and the calculations of Section~\ref{sectioncompute} to study the structure of $3$--dimensional bordism groups.  We let $\Omega$ denote either the $3$--dimensional rational Spin--cobordism group, $\Omega^{\Q}_{\text{Spin}}$, or the  $3$--dimensional $\Z/2\Z$--homology bordism group $\Omega^{\Z/2\Z}$.  In each case, $\Omega$ contains a subgroup $\Omega_T$, defined to be the kernel of the homomorphism from $\Omega$ to the corresponding topological cobordism group.  It is also the case that $\Omega$  contains the subgroup $\Omega_I$ generated by homology spheres.  By Freedman~\cite{fr}, $\Omega_I \subset \Omega_T$, and by Furuta~\cite{furuta:cobordism}, $\Omega_I \subset \Omega^{\Z/2\Z}$ is infinitely generated.   Our techniques show that $\Omega_T / \Omega_I$ is infinitely generated.  Of course it follows  that  $ \Omega^{\Q}_{\text{Spin}}$ is infinitely generated, which apparently was not previously known.  As an additional consequence, we show the existence of 
$4$--dimensional rational homology balls that do not have pseudo-handlebody structures.
 
In Section~\ref{sectionboundary} we take up the question of boundary links and link concordance, again  with regard to the difference between smooth and topological cobordisms.  Cochran and Orr~\cite{cochran-orr:boundary-links} constructed $2$--component links in $S^3$ that are not concordant to boundary links, despite the fact that all of their Milnor $\bar{\mu}$--invariants vanish.  Generalizing observations made in~\cite{livingston:boundary-links}, we construct   further examples: in this case the links   have the remarkable property that they are topologically concordant to boundary links (and thus all Milnor and Cochran-Orr invariants vanish) and yet they are not smoothly concordant to boundary links.

 \vskip.1in
 \noindent{\bf Acknowledgments}  In addition to the support of the NSF, this collaboration was facilitated by generous support by our home institutions as well as the University of California, Berkeley, and MIT.  We also thank Tim Cochran, Stefan Friedl, Peter Horn, Rob Kirby,  Tom Mrowka, Nikolai Saveliev, and Peter Teichner for several useful and informative conversations. 

 
\section{Spin$^c$  structures and metabolizers for linking forms}\label{sectionspinc}
 
Our concordance obstruction will come from an invariant of $\SpinC$ $3$--manifolds.   For these purposes, the only relevant $\SpinC$ structures  on a given $3$--manifold, $Y$, are those which could extend over a putative $4$--manifold $W$ satisfying $\partial W=Y$. In this section, we clarify this extension problem by connecting it with metabolizers for the linking form on $Y$.

Denote by $\SpinC(W)$ the set of Spin$^c$  structures on a manifold, $W$.  
If $\SpinC(W)$ is non-empty, then it admits a free, transitive action by $H^2(W;\Z)$, which we denote $z \cdot \spinc$ for 
$z\in H^2(W;\Z),\ \spinc \in \SpinC(W)$.  The action affects the first Chern class by the rule $c_1(z \cdot \spinc ) - c_1(\spinc ) = 2z$. 
Thus if $H^2(W;\Z)$ has no $2$--torsion, then the first Chern class establishes an injection 
$$c_1 \co \SpinC(W) \to H^2(W)$$ 
with image the set of classes that restrict to $w_2(W) \pmod{2}$.  
This map is natural with respect to restriction: if $Y \subset W$ then  $c_1(\spinc |_Y) = i^*c_1(\spinc )$, where $i^*$ is the restriction map on cohomology.  

Let us now make the further simplifying assumptions that $H_2(W; \Z / 2\Z) = 0$ and $H_1(\partial W; \Z / 2\Z) = 0$.  Under the Lefschetz and Poincar\'e duality isomorphisms $H^2(W) \cong H_2(W, \partial W)$ and $H^2(\partial W) \cong H_1(\partial W)$, the set of Spin$^c$ structures on $\partial W$ that arise as the restriction of Spin$^c$  structures on $W$ correspond to $$\mathrm{Im}( H_2(W, \partial W)\overset{\partial}{\rightarrow}H_1(\partial W)).$$  This is precisely the kernel of the map $i_*\co H_1(\partial W) \to H_1(W)$, induced by inclusion.  Understanding this  kernel can be aided significantly by an observation of Casson and Gordon~\cite{cg}.

To state their observation, recall that a subgroup  $M \subset H_1(Y)$ is called a {\it metabolizer} if \begin{itemize} \item $|M|^2 = |T_1(Y)|$, where $T_1$ denotes the torsion subgroup of $H_1(Y;\Z)$, and 
 \item The $\Q/\Z$--valued linking form on $H_1(Y;\Z)$ is identically zero on $M$.
 \end{itemize}
In these terms, we have the following ~\cite{cg}.
 
 \begin{prop}\label{prop:cassongordon} If $W$ is a compact $4$--manifold with $H_2(W;\Z/2\Z) = 0$ and $H_1(\partial W; \Z / 2\Z) = 0$,  then the set of Spin$^c$  structures on $\partial W$ that extend to Spin$^c$ structures on $W$ correspond via $c_1$ and Poincar\'e duality to a metabolizing subgroup $M \subset H_1(\partial W)$.
  \end{prop}

In the special case that $Y$ is a $\Z/2\Z$--homology $3$--sphere\footnote{Throughout we will say that an $n$-manifold is an $R$--homology sphere (ball) if it has the same homology with $R$--coefficients as the $n$--sphere (ball).}, then it admits a unique spin structure; denote by $\spinc_0$ the Spin$^c$ structure corresponding to this spin structure.  This notation extends to the other $\SpinC$ structures as follows.

\begin{definition} For $Y$ a $\Z/2\Z$--homology sphere, and $m\in H_1(Y;\Z)$, let $\spinc_m$ be the unique Spin$^c$ structure satisfying $PD[c_1(\spinc_m)]=2m$. \end{definition}

\noindent{\bf Comments.}  In many cases, we will consider manifolds with $H_1(Y)$ cyclic.  In these cases we refer to $\SpinC$ structures by $\spinc_m$, with $m\in \Z/p\Z$.   The factor of two appears in our notation in order to simplify some of the exposition regarding the Heegaard Floer complexes that appear later.  Note, however, that for $\Z/2\Z$ homology spheres, multiplication by two is an isomorphism of $H_1(Y)$ that leaves all subgroups invariant.

 
\section{The Heegaard Floer correction term $d$ as a concordance obstruction}\label{sectiondandobstruct}
 
 In this section, we define an invariant that serves as an obstruction for a knot to be concordant to a knot with Alexander polynomial one.  We then proceed to determine sufficient conditions, in terms of our obstruction, for a collection of knots to be linearly independent (over $\Z$) in the quotient group $\calc_T / \calc_\Delta$; see Theorem \ref{independence}.   In light of Proposition \ref{prop:cassongordon}, it is perhaps unsurprising that deriving these conditions requires some analysis of linking forms and their metabolizers.  We relegate the bulk of this algebra to the appendix.

To a $\SpinC$ $3$--manifold, $(Y,\spinc )$, \ons \ associate a rational-valued invariant, $d(Y,\spinc)\in\Q$, called the {\em correction term}.   The definition comes from grading information in their Heegaard Floer homology theory, and is reviewed in Section~\ref{sectionhf}.  For now, we need only the following two properties, which correspond to theorems 4.3 and 1.1 of~\cite{os2}, respectively.
\begin{enumerate}
\item (Additivity) $d(Y\#Y',\spinc\#\spinc')=d(Y,\spinc)+d(Y',\spinc')$; that is, $d$ is additive under connected sums.
\item (Vanishing) Suppose $(Y,\spinc ) = \partial(W, \spinct)$, where $W$ is a $\Q$--homology ball and $\spinct$ is a Spin$^c$  structure on $W$ that restricts to $\spinc$ on $Y$.  Then $d(Y,\spinc) = 0$. 
\end{enumerate}
 
\noindent Our obstruction is defined as a difference of correction terms.

 \begin{definition}\label{dbardef} For $Y$ a $\Z/2\Z$--homology sphere, define $\bar{d}(Y,\spinc ) = d(Y,\spinc ) - d(Y,\spinc_0)$, where $d$ is the \os \ correction term  and $\spinc_0$, as above, is the unique spin structure on $Y$. \end{definition}
  
We have the following vanishing theorem for $\bar{d}$.   
\begin{theorem}\label{T:qhs} Let   $P$ be a finite set of (distinct) odd primes.   Suppose that  $W$ is a $\Z/2\Z$--homology $4$--ball and  $\partial W = \#_{p \in P} Y_{p}\  \#\ Y_1$, where 
\begin{itemize}
\item $p^k H_1(Y_p) = 0$ for each $p \in P$ and some $k> 0$.
\item $Y_1$ is a $\Z$--homology $3$--sphere. 
\end{itemize}
Then for each $p \in P$, there is a metabolizer $M_p \subset H_1(Y_p)$ for which $\bar{d}(Y_p, \spinc_{m_p}) = 0$
for all $m_p \in M_p$.
  \end{theorem}
  
   \begin{proof}
   Proposition \ref{prop:cassongordon}, together with the vanishing property of the correction terms, shows that 
 $H_1(\partial W)$ possesses a metabolizer $M$, satisfying   
 $$d(\partial W, \spinc_m) = 0 \text{\ for\ all\ } m \in M.$$
 There is a decomposition of $M$ as a direct sum $\oplus_{p\in P} M_p$, where $M_p$ is a metabolizer for
  $H_1(Y_p)$ and, in particular, is $p$--torsion.

   For each $p\in P$ we can gather all but the $p$--summand of $\partial W$ to write $\partial W = Y_p \ \# \ Z_p$.  Now, given $m_p \in M_p$, consider  $m = m_p \oplus 0 \in H_1(\partial W) = H_1(Y_p) \oplus H_1(Z_p)$.  Then $d(\partial W, \spinc_{m_p \oplus  0}) = 0$ by the considerations above.   Additivity shows that $$d(Y_p, \spinc_{m_p}) + d(Z_p, \spinc_0) = 0.$$  But this holds for any $m_p\in M_p$; that is, $d(Y_p, \spinc_{m_p})$ is independent of the choice of $m_p \in M_p$.  It follows that $\bar{d}(Y_p, \spinc_{m_p}) = 0$ for all $m_p \in M_p$, as desired.

  \end{proof}
  
By using branched covers, the theorem yields the desired concordance obstruction.
  
  \begin{corollary}\label{C:delta1} Let $K \subset S^3$ be a knot with $p^k H_1(\Sigma(K)) = 0$ for some $k$, where $\Sigma(K)$ is the $2$--fold branched cover of K.  Suppose that $K$ is concordant to a knot $K'$, satisfying $\Delta_{K'}(t)=1$.  Then there exists a metabolizer $M\subset H_1(\Sigma(K))$ such that $\bar{d}(\Sigma(K), \spinc_m) = 0$, for all $m\in M$.
  
  \end{corollary}
  
  \begin{proof} Let $W$ be the $2$--fold branched cover of $B^4$ branched over a slice disk for $K\ \# \ -K'$.  Then $\partial W = Y_p \ \# \ Y_1$ where $Y_p = \Sigma(K)$  and $Y_1 = \Sigma(-K')$.  Since $W$ is a $\Z / 2\Z$--homology ball and $Y_1$ is a homology sphere (as follows from  $\Delta_{K'}(t) = 1$) the previous theorem applies. \end{proof}

  \begin{corollary} Let $K = \#_{p \in P}K_p  \cs  K_1$ be a connected sum of knots satisfying
   \begin{itemize} 
   \item $p^k H_1(\Sigma(K_p)) = 0$ for each $p$ in a set of primes, $P$, and some $k$, 
   \item $H_1(\Sigma(K_1)) = 0$.
   \end{itemize}
   Suppose $K$ is slice. Then for each $p \in P$, there is a metabolizer 
   $M_p \subset H_1(\Sigma(K_p))$ for which $\bar{d}(\Sigma(K_p), \spinc_{m_p}) = 0$ for all $m_p \in M_p$.
  
  \end{corollary}
  
  \begin{proof} The proof is much like the previous argument, grouping together all $\Sigma(K_q)$ for $q \ne p$, since this space will be a $\Z / p\Z$--homology sphere.
  
   \end{proof}

In the next section, we construct a family of knots, $\{K_p\}$, which we would like to show are linearly independent in $\calc_T / \calc_\Delta$.  By definition, this means that no $\Z$--linear combination $K=\Sigma \ n_pK_p$ is equal to zero or, equivalently, is concordant to a knot with Alexander polynomial one.  Note the sum is in concordance; for instance, $-3K$ means the connected sum of three copies of the mirror of $K$ with reversed orientation.

In pursuit of such independence, the previous corollary reduces the problem to showing that for every $n \ne 0$ and metabolizer $M \subset H_1(\Sigma(nK_p))$, there is some $m \in M$ with $\bar{d}(\Sigma(nK_p), \spinc_m) \ne 0$.  In order to reduce this further, to the case that $n = 1$, we have the following theorem.  The proof (compare~\cite{livingston-naik:4-torsion,livingston-naik:torsion}) represents a significant algebraic detour, and is left to the appendix.

\begin{theorem} Suppose $p$ is a prime satisfying  $p \equiv 3 \mod 4$.     If $K $ satisfies $H_1(\Sigma(K)) = \Z/p^2\Z$ and there is a metabolizer for $M \subset  H_1(\Sigma(nK))$ for which $\bar{d}(\Sigma(nK), \spinc_m) = 0$ for all $m \in M$, then $\bar{d}(\Sigma(K), \spinc_{pk}) = 0 $ for all $k$.

\end{theorem}

Combining this with the work above, we immediately have the following theorem. 

\begin{theorem} \label{independence} Suppose that $\{K_p\}$ is a collection of knots  indexed by the set of primes $p \equiv 3 \mod 4$.  Suppose further that  $H_1(\Sigma(K_p)) = \Z/ p^2\Z$, and that for each $p$,  $\bar{d}(\Sigma(K_p), \spinc_{pk}) \ne 0$ for some $k$. Then no $\Z$--linear combination of the knots in $\{K_p\}$ is concordant to a knot with trivial Alexander polynomial.

\end{theorem}

In Section~\ref{sectioncompute} we show that a family of topologically slice knots constructed in the next section satisfy the hypotheses, thus  demonstrating the truth of Theorem~\ref{infgen}.

 
\section{The knots $K_p$.}  \label{sectionknots}
 
 In this section, we construct an infinite family of topologically slice knots.   These knots will be used with Theorem \ref{independence} to prove Theorem \ref{infgen}. The details of the construction were  motivated by a desire to find knots whose $2$--fold branched covers are realized by surgery on knots in $S^3$ with computable Floer invariants.  As discussed in Section \ref{sectionhf}, these knot Floer homology invariants can be used to determine the correction terms.

   Figure~\ref{figure1} illustrates a knot $K_3$ in the family of knots $\{K_p\}$, where $p$ is a prime satisfying $p \equiv 3 \mod 4$ and  $J_p$ is the connected sum of $(3p-1)/4$ positive-clasped, untwisted Whitehead doubles of the right-handed trefoil knot.  The orientation-preserving band in the non-orientable  surface $F_p$ bounded by $K_p$ is untwisted and has the knot $J_p$ tied in it.
   
 \begin{prop} $K_p$ is topologically locally flat slice.
  \end{prop}
   
\begin{proof} The surface $F_p$ is a punctured Klein bottle.  The core of the left band of $F_p$  is a simple closed curve $\alpha$, which represents the knot $J_p$.  Since the neighborhood of $\alpha$ is an untwisted annulus and $J_p$ is topologically slice by Freedman's theorem (untwisted doubles have trivial Alexander polynomial), $F_p$ can be surgered in $B^4$ along $\alpha$.  Performing this surgery on the punctured Klein bottle yields a disk.

\end{proof}

 \noindent For $q\in \Z$, let $S^3_q(K)$ denote the manifold obtained by $q$--surgery on a knot $K \subset S^3$.
   
   \begin{prop} \label{surgerydiagram}The $2$--fold branched cover $ \Sigma(K_{p})  = S^3_{p^2}( 2J_p  \cs T_{p-1,p})$, where $T_{p-1,p}$ is the $(p-1,p)$--torus knot.  In particular, 
$H_1(\Sigma(K_{p})) = \Z / p^2 \Z$, and hence  $K_p$ has non-trivial Alexander polynomial.
\end{prop}

\begin{proof} 
  According to~\cite{akbulutkirby}, the $2$--fold branched cover of $K_p$, $\Sigma(K_p)$, is given by the surgery diagram illustrated in Figure~\ref{figure2}.  There are $-p$ full twists between the components of the $2$--component link shown, and the surgery coefficients are 0 and $-1$. The notation $J^r$ denotes the knot $J$ with its string orientation reversed.  Since doubled knots are reversible, in our case $J^r = J$.

\begin{figure}[ht]
\psfrag{0}{$0$}
\psfrag{m1}{$-1$}
\psfrag{mp}{$-p$}
\psfrag{jp}{$J_p$}
\psfrag{jpr}{$J_p^r$}
\fig{.5}{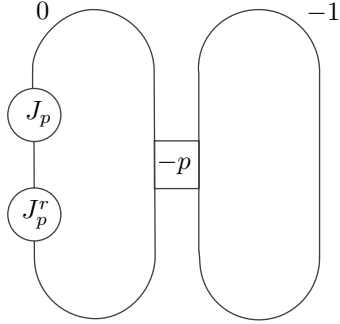}
\caption{Surgery diagram of $\Sigma(K_p)$.}\label{figure2}
\end{figure}

If an unknotted component of a surgery diagram of a $3$--manifold has framing $-1$, that component can be removed, with the effect of putting a full twist in the curves that pass through it and increasing their framings by the square of the linking number with the unknotted component.  (This procedure is referred to as {\it blowing down} the $-1$.)  In the present case, the result is the surgery diagram given in Figure~\ref{figure3}.

\end{proof}
\begin{figure}[h]
\psfrag{p2}{$p^2$}
\psfrag{tp}{$T_{p-1,p}$}
\psfrag{jp}{$J_p$}
\psfrag{jpr}{$J_p^r$}
\fig{.5}{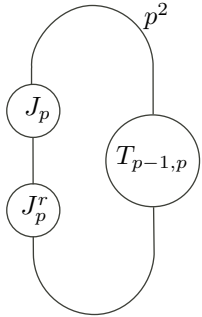}
\caption{Surgery diagram of $\Sigma(K_p)$ as $p^2$--surgery on $T_{p-1,p} \cs J_p \cs J_p^r$.}\label{figure3}
\end{figure}

The result of~\cite{os1} that we will use to compute  $ d (S^3_q(K), \spinc_m)$ requires that $q \ge 2g_3(K) -1$. The following lemma verifies that that this condition is satisfied.
\begin{lemma} The $3$--genus of $T_{p-1,p} \cs J_p \cs J_p^r$ is  $ \frac{p^2 +1}{2}$.  
\end{lemma}

\begin{proof} We have $g_3(T_{p-1,p}) = \frac{(p-2)(p-1)}{2}$.  Since  $J_p = \frac{3p-1}{4} D(T_{2,3})$ and $g_3(D(T_{2,3})) = 1$, we have 
$$g_3(T_{p-1,p} \cs J_p \cs J_p^r) =  \frac{(p-2)(p-1)}{2} + \frac{3p-1}{2} = \frac{p^2 +1}{2}.$$
\end{proof}


\section{Background on Heegaard Floer homology}\label{sectionhf}
 
In this section, we collect some basic facts about \ons's  Heegaard Floer homology invariants.   Our main purpose is
to introduce the algebraic structures inherent in the theory.  These structures will subsequently be exploited, both to
define the correction terms used for our concordance obstruction and to aid in its calculation.  Details regarding the 
invariants used here can be found in \cite{os3,os2,Knots}.  Throughout,  we let $\F=\Z/2\Z$ denote the field with $2$ 
elements.  Many of the chain complexes and homology groups in what follows have gradings; these are indicated by subscripts when we discuss groups in a single grading and are omitted otherwise.
  
 \subsection{The chain complexes and the definition of $d$} In \cite{os3}, \ons \ associate various chain complexes to a pair, $(Y,\spinc)$, consisting of an oriented $\Q$--homology sphere, $Y$, and a $\SpinC$ structure $\spinc$ (invariants are defined for arbitrary $3$--manifolds, but those  of rational homology spheres will be sufficient for our applications).    As input,  the theory takes a pointed Heegaard diagram for $Y$ consisting of a surface $\Sigma$ of genus $g$, together with two $g$--tuples of attaching curves, $\vec{\bf \alpha}$, $\vec{\bf \beta}$, and a distinguished basepoint in their complement, $z$.  By taking the $g$--fold symmetric product of the diagram one arrives at a $2g$--dimensional (complex) manifold, together with two $g$--dimensional submanifolds, denoted $\Ta,\Tb$.  The basepoint leads to a complex hypersurface, $V_z$, consisting of those unordered $g$--tuples of points on $\Sigma$, at least one of which is $z$. 
 
 The most general Heegaard Floer complex is denoted $\CFinf(Y,\spinc)$.  It is generated by pairs, $[\x,i]$, where $\x\in \Ta\cap \Tb$ is an intersection point, $i\in \Z$ is an integer, and $\spinc_z(\x)=\spinc$ (the basepoint induces a map $\spinc_z\co \Ta\cap \Tb\rightarrow \SpinC(Y)$).   Roughly speaking, the boundary operator counts pseudo-holomorphic disks in the symmetric product that connect $\x$ to $\y$.  The integer keeps track of the algebraic intersection number of such disks with the hypersurface, $V_z$.   By identifying $[\x,i]$ with $U^{-i}\cm\x$, the chain groups can be thought of more algebraically as the free $\F[U,U^{-1}]$ module generated by intersection points of $\Ta$ and $\Tb$.  Here $U$ is a formal polynomial variable, and under this correspondence we have $U^i\cm [\x,j]=[\x,j-i]$.   The complex is relatively $\Z$--graded by the formula  $$ gr([\x,i])-gr([\y,j])= \mu(\phi)-2(i-j),$$
 where $\mu(\phi)$ denotes the Maslov index of any Whitney disk connecting $\x$ to $\y$.   With this formula, it is clear that the variable $U$ (respectively $U^{-1}$) carries a grading of $-2$ (respectively $2$).  Moreover, the complex can be endowed with an absolute grading, which takes values in $r+\Z$, where $r\in \Q$ is a fixed (dependent only on $\spinc$) rational number.  This grading, which we denote $gr$, is defined in the spirit of index theory for $4$--manifolds with boundary, \cite{aps:I} and considers characteristic numbers of a $\SpinC$ cobordism between $(Y,\spinc)$ and the $3$--sphere.  For details on the absolute grading, see \cite{HolDiskFour,os2}.

We denote the homology of the above complex by $\HFinf(Y,\spinc)$.  By itself, this invariant is rather uninteresting, as indicated by the following theorem

\begin{theorem}\cite[Theorem 10.1]{oz:hf-properties}\label{infinity}
Let $Y$ be a rational homology three sphere.  Then $$\HFinf(Y,\spinc)\cong \F[U,U^{-1}],$$
 for any $\SpinC$ structure, $\spinc$.  
 \end{theorem}

The theory becomes more interesting when one notices that $\CFinf(Y,\spinc)$ has a distinguished subcomplex, consisting of pairs $[\x,i]$ with $i<0$.  We denote this subcomplex by $\CFm(Y,\spinc)$.  That this is a subcomplex follows from the fact that pseudo-holomorphic disks intersect $V_z$ positively, when transverse.    We have the corresponding short exact sequence
\begin{equation}\label{eq:ses} 0\rightarrow \CFm(Y,\spinc) \rightarrow \CFinf(Y,\spinc) \rightarrow \CFinf(Y,\spinc)/\CFm(Y,\spinc) \rightarrow 0.\end{equation}
The quotient complex, which we henceforth denote by $\CFp(Y,\spinc)$, is generated by pairs $[\x,i]$ with $i\ge 0$.

There is a fourth complex, denoted $\CFa(Y,\spinc)$, which frequently appears.  It can be described as the complex  ker$\{\CFp\overset{\cm U}\rightarrow \CFp\}$.  Perhaps more concretely, it is the complex generated by pairs $[\x,0]$, whose boundary operator counts holomorphic disks that miss the hypersurface.

Theorem $1.1$ of \cite{os3} indicates that the  homology of all four of these chain complexes  are invariants of the pair, $(Y,\spinc)$.  That is, they are independent of the many choices involved in the construction: for example the Heegaard diagram, the almost complex structure on the symmetric product, and the basepoint.   Note that all four groups are naturally modules over $\F[U]$, and the module structure is also an invariant of $(Y,\spinc)$.   Of course it follows from the definitions that $\HFinf$ is also a module over $\F[U,U^{-1}]$, and that the module structure on $\HFa$ is trivial (in the sense that $U$ acts as zero).

Theorem \ref{infinity} allows the definition of a numerical invariant, the so-called ``correction term" or $d$--invariant, of a $\SpinC$ $3$--manifold.  Appropriately interpreted, this invariant serves as our concordance obstruction (see Definition \ref{dbardef} above).

\begin{definition}$$d(Y,\spinc)= \underset{\alpha\ne0\in \HFp(Y,\spinc)}{\text{min}}  \{gr(\alpha)\ |\ \alpha\in \text{Im} \ U^k, \text{\ for\ all\ } k\ge 0 \}$$

\end{definition}
Note that  $d(Y,\spinc)$ is a rational number, in general.  The additivity and vanishing properties of $d$  mentioned in 
Section  \ref{sectiondandobstruct}  follow immediately from Theorems $4.3$ and $1.1$ of \cite{os2}, respectively.

  We conclude this subsection with an example.
\bigskip

\noindent{\bf Example: the $3$--sphere.}    Examining the standard genus one Heegaard diagram for the $3$--sphere allows one to compute its Floer homology directly.  We have the following isomorphisms of $\F[U]$ modules, where the grading of the element $1\in \F[U]$ is equal to $0$, and $U$ carries (as above) a grading of $-2$.  Note that $S^3$ carries a unique $\SpinC$  structure, so we suppress this from the notation.
$$
\begin{array}{l}
\HFinf(S^3)\cong \F[U,U^{-1}]\\
\HFp(S^3)\cong \F[U,U^{-1}]/U\cm \F[U]\\
\HFm(S^3)\cong U\cm \F[U]\\
\HFa(S^3)\cong \F[U]/U\cm \F[U]\cong \F\\
\end{array}
  $$

 Moreover, the long exact sequence of $\F[U]$--modules coming from the short exact sequence \eqref{eq:ses} becomes
 $$ 0\rightarrow U\cm\F[U]\rightarrow \F[U,U^{-1}] \rightarrow \F[U,U^{-1}]/U\cm \F[U]\rightarrow 0,$$ 
 and the map $\HFa(S^3)\rightarrow \HFp(S^3)$ is injective. 
 
 Note that $1\in \HFp(S^3)$ is in the image of $U^k$ for all $k\ge0$, and has grading $0$. Thus $d(S^3)=0$.

\subsection{Knot Floer homology and surgery on knots}\label{subsec:HFKbackground}  The discussion of the last section implies that $\CFinf(Y,\spinc)$ is naturally a $\Z$--filtered chain complex.  Indeed, the  subcomplexes $U^d \cm \CFm(Y,\spinc)$  are the corresponding terms in this filtration.  Equivalently, the filtered subcomplexes are those generated by pairs $[\x,i]$ satisfying $i<-d$, where $d\in \Z$. 

 A knot $K\subset Y$ induces a second filtration on $CF^\infty(Y,\spinc )$, whose construction we briefly describe.    For our purposes, it will be sufficient to consider the case of knots in the $3$--sphere, $K\subset S^3$, and henceforth we deal exclusively with this special case.

 Given $K\subset S^3$, the second filtration of $\CFinf(S^3)$ arises by consideration of a doubly-pointed Heegaard diagram adapted to the knot.  This is a Heegaard diagram with basepoints $z$ and $w$ for which $K$ can be realized as the union of two arcs $t_\alpha\cup t_\beta$, where $t_\alpha$ (respectively $t_\beta$)  is properly embedded in the handlebody specified by $\vec{\bf \alpha}$ (respectively $\vec{\bf \beta}$), and both arcs have common boundary consisting of $z\cup w$.    
 
The data above allows us to define a complex $\CFKinf(K)$, freely generated as an $\F[U,U^{-1}]$--module by triples $[\x,i,j]$, $i,j\in \Z$ satisfying a homotopy-theoretic constraint: $$\langle c_1(\underline\spinc(\x)),[S]\rangle +2(i-j)=0.$$
Here $c_1(\underline\spinc(\x))$ is the Chern class of a $\SpinC$  structure associated to $\x$ on $S^3_0(K)$ (the manifold obtained by zero surgery on $K$), and $[S]\in H_2(S^3_0(K);\Z)$  is the class that arises from extending a Seifert surface by the meridional disk of the surgery torus.   The boundary operator on $\CFKinf(K)$ is defined as before, except that now the second index keeps track of the intersection number of holomorphic disks with the hypersurface $V_w$, specified by the additional basepoint.  Forgetting $j$, we are left with $\CFinf(S^3)$, and  positivity of intersections ensures that the projection $[\x,i,j]\rightarrow j$ provides it with a second $\Z$--filtration. 

Thus  $\CFKinf(K)$ is a $\Z\oplus\Z$--filtered chain complex; that is, a chain complex $C_*$, together with a map $\Filt: C_*\rightarrow \Z\oplus\Z$ satisfying $\Filt(\partial\x)\le \Filt(\x)$, where $\le$ is the standard partial order on $\Z\oplus \Z$.  Theorem $3.1$ of \cite{Knots} shows that the $\Z\oplus \Z$--filtered chain homotopy type of $\CFKinf(K)$ is an invariant of the isotopy class of $K$.

Much of the power of an invariant that takes values in the  $\Z\oplus\Z$--filtered chain homotopy category  lies in our ability to derive further invariants by considering the homology of sub and quotient complexes.  For instance, we can consider the subcomplex $$C_*\{\mathrm{max}(i,j-m)<0\}\subset \CFKinf(K)$$ 
generated by triples $[\x,i,j]$ satisfying max$(i,j-m)<0$.  The homology of this subcomplex is also an invariant of $K$.   The corresponding quotient complex is generated by triples satisfying max$(i,j-m)\ge0$, and is denoted $C_*\{\mathrm{max}(i,j-m)\ge0\}$.  We will suppress $K$ from the notation for complexes derived from $\CFKinf(K)$ whenever the particular knot is clear from the discussion.

In the present context, the importance of these knot invariants comes from the following theorem, which indicates that they can be identified with the Floer homology of manifolds obtained by surgery on $K$. To make this precise,  denote by $\spinc_m\in \SpinC(S^3_q(K))$ the unique $\SpinC$  structure  that extends over the $2$--handle cobordism  induced by $q$--surgery on $K$, to a $\SpinC$  structure, $\spinct_m$ satisfying:
$$\langle c_1(\spinct_m),[S]\rangle + q = 2m.$$
In terms of this labeling of $\SpinC$   structures, we have

\begin{theorem}\cite[Theorem $4.4$]{Knots}\label{thmcompute} Let $K \subset S^3$ be a knot of Seifert genus $g_3(K) = g$, and let $q$ be a positive integer such that $q \ge 2g-1$.  Then for all  $m$ satisfying $|m | \le   \frac{1}{2}(q-1)$, we have a chain homotopy equivalence of graded complexes over $\F[U]$,
$$CF^+_*(S^3_q(K), \spinc_{m}) \simeq   C_{*+s(q,m)}\{ \mathrm{max}(i,j-m)\ge0\},$$

\noindent where the grading shift $s(q,m)$ is given by the following formula 
$$s(q,m) = \frac{  -{(2m-q)^2} +q}{4q}.$$ 

\end{theorem}

\section{Computing $\bar{d}(S^3_{p^2}(T_{p-1,p} \cs J_p \cs J_p^r)$}\label{sectioncompute}
 In this section we turn to the computation of the concordance obstruction for our family of knots.  Having identified the branched double cover of these knots with the manifolds obtained by $p^2$ surgery on the knot $$L_p := T_{p-1,p} \cs J_p \cs J_p^r  = T_{p-1,p} \cs  \frac{3p-1}{2} D(T_{2,3})$$ (see Proposition~\ref{surgerydiagram}) we will accomplish this task by analyzing $\CFKinf(L_p)$ and using Theorem \ref{thmcompute} to extract the correction terms necessary for $\bar{d}$.

Before going further, we describe some aspects of the computation in more detail.  Theorem \ref{thmcompute} allows us to  compute $\HFp(S^3_{p^2}(L_p),\spinc)$ completely in terms of $\CFKinf(L_p)$. Furthermore, a K{\"u}nneth theorem for knot Floer homology says that $\CFKinf(K_1\#K_2)\cong \CFKinf(K_1)\otimes \CFKinf(K_2)$.  Thus it suffices, in principle, to know $\CFKinf(D(T_{2,3}))$ and  $\CFKinf(T_{p-1,p})$.  The former invariant was studied in \cite{Doubling}, while the latter was determined in \cite{oz:hfk-lens}. While this appears to complete the picture, two issues make the situation more subtle. 

The first issue is that the size of the chain complex grows very quickly with $p$; a minimal generating set for $\CFKinf(L_p)$ as an $\F[U,U^{-1
}]$ module consists of $(2p-3)(15)^{\frac{3p-1}{2}}$ elements.  Such a complex is somewhat unwieldy to work with in the context of Theorem \ref{thmcompute}.   The second issue is that the results of \cite{Doubling} leave an ambiguity in the full nature of $\CFKinf(D(T_{2,3}))$: the results of \cite{Doubling} only determine differentials in $\CFKinf$ that drop one or the other of the filtration indices, and not differentials that drop both.

In light of this, we found it convenient to distill only the properties of $\CFKinf(L_p)$ necessary for the computation of a single  $\bar{d}$ invariant which, by Theorem \ref{independence}, is sufficient for the topological applications.  As it turns out, this requires far less information than the full $\Z\oplus\Z$--filtered chain homotopy type of $\CFKinf(L_p)$, and we hope that similar methods can be exploited to compute the correction terms in other situations when chain complexes become complicated or are only partially known.

\subsection{\bf Understanding $CFK^\infty(T_{p-1,p})$}  With the general strategy in place, we begin by studying the  complex of the torus knot $T_{p-1,p}$.  The filtered chain homotopy type of this complex is determined by the main theorem  in \cite{oz:hfk-lens}.  For the specific case of $T_{p-1,p}$, the complex can also easily be understood from the definition, as $T_{p-1,p}$ admits a genus one doubly-pointed Heegaard diagram.  For such knots, the methods developed in Section $6.2$ of \cite{Knots} (and further exploited in \cite{goda-matsuda-morifuji:hfk11}) allow one to compute the differential on $\CFKinf$ via the Riemann mapping theorem.

For our purposes, it will be most convenient to use Theorem $1.2$ of \cite{oz:hfk-lens} to understand the structure of $\CFKinf(T_{p-1,p})$.   While the full $\Z\oplus\Z$--filtered chain homotopy type of $\CFKinf(T_{p-1,p})$ can be determined from this theorem, it is stated as a result about the knot Floer homology groups.  Recall that these groups, denoted $\HFKa_*(K,j)$ are the associated graded groups of the subquotient complex \newline $C\{i=0\}\subset\CFKinf(K)$, equipped with the $\Z$--filtration $[\x,0,j]\rightarrow j$.

To state the theorem, let $$\Delta_K(t)=\sum_{j=-g}^{j=g} a_j \cm t^j$$ denote the  Alexander polynomial, normalized so that $a_j=a_{-j}$ and $\Delta_K(1)=1$.  Let $$n_{-k}<...<n_k,$$
denote the sequence of integers, $j$, for which $a_j\ne 0$.  Under the assumption that some positive framed surgery on $K$ produces a lens space (or, more generally, an L-space), Theorem $1.2$ of \cite{oz:hfk-lens}  indicates that this sequence determines the knot Floer homology groups.

More precisely, we have $\HFKa(K,j)=0$ unless $j=n_s$ for some $s$,  in which case
$\HFKa(K,n_s)\cong \F$.  Moreover, the homological grading of $\HFKa(K,n_s)$ is given by an integer $\delta_s$, determined by the formulas (for $l\ge 0$): 
\begin{eqnarray} \delta_{k-2l}&=& -2 \sum_{j=0}^{2l-1} (-1)^j \cm n_{k-j},\\
\delta_{k-2l-1}&=& \delta_{k-2(l+1)}  + 1.\end{eqnarray}
Note that we have expressed \ons's recursive formula for $\delta_i$ in closed form, and that $\delta_k=0$ since the summation in this case is vacuous.

We now use this theorem to extract the properties of $\CFKinf(T_{p-1,p})$ needed for our application.   To be more precise, when we refer to {\em the} chain complex $\CFKinf(K)$ of a given knot, we really refer to the $\Z\oplus\Z$--filtered chain homotopy type of $\CFKinf(K)$.  As such, we will always work with a representative for this type that is {\em reduced}, in the sense of \cite[Section 4]{RasThesis}.  This means that the differential on $\CFKinf(K)$ strictly lowers the $(i,j)$ filtration.   The existence of such a representative for any $\Z\oplus\Z$--filtered chain homotopy type follows in exactly the same manner as the proof of \cite[Lemma 4.5]{RasThesis}.

For the reader unfamiliar with the Floer homology of torus knots, it may be enlightening to skip the proof of the following proposition on first reading and proceed to the discussion immediately following it.   There, we give a more conceptual description of the chain complexes that guided the statement and proof of the proposition.  
\begin{prop}\label{gradinglemmaprime} Consider the chain complex $\CFKinf_*(T_{p-1,p})$, for $p$ odd.  Then 
\begin{itemize}
\item Any chain $[\x,i,j]\in \CFKinf_0(T_{p-1,p})$  satisfies $i+j\ge \frac{p^2-2p+1}{4}$.
\item Any chain $[\x,i,j]\in \CFKinf_1(T_{p-1,p})$  satisfies $i+j\ge \frac{p^2-1}{4}$.
\item There exists a cycle $[\x, \frac{p^2-4p+3}{8}, \frac{p^2-1}{8}]$ that is homologous to a generator of $\HFinf_0(S^3)\cong \F$.
\end{itemize}
\end{prop}

\begin{proof}
Since $p^2-p+1$ surgery on $T_{p-1,p}$ is a lens space \cite{Moser1971}, we can employ \ons's theorem to compute the knot Floer homology groups.  To begin, recall that  $$\Delta_{T_{p-1,p}}(t)=t^{-g}\cm\frac{(t^{(p-1)p}-1)(t-1)}{(t^{p-1}-1)(t^{p}-1)},$$
where $g=\frac{(p-2)(p-1)}{2}$ denotes the Seifert genus of $T_{p-1,p}$.  This can be rewritten as
$$\Delta_{T_{p-1,p}}(t)=  t^{g}\cm\lbrace 1-  \sum_{l=1}^{p-2}    t^{-p(l-1)-1} + \sum_{l=1}^{p-2} t^{-l(p-1)}\rbrace . $$  (To demonstrate  this equality, multiply by $(t^{p-1}-1)(t^p-1)$.  The first sum, when multiplied by $(t^{p} -1)$, becomes telescoping and collapses, as does the second sum when multiplied by $(t^{p-1}-1)$.)

It follows that the sequence of integers corresponding to $t$ powers with non-vanishing coefficient can be expressed as the union of two sequences:
$$
\begin{array}{lll} n_{k-2l} & =  g- l(p-1) & \ \  l=0,...,p-2\\
n_{k-2l-1} &=  g- lp-1 & \ \ l=0,...,p-3,
\end{array}
$$
Here, $k$ is easily seen to equal $p-2$, as there are $2p-3$ non-vanishing coefficients in $\Delta_{T_{p-1,p}}(t)$.

Straightforward algebra determines the gradings $\{\delta_s\}$ from the $\{n_s\}$:
$$
\begin{array}{lll} \delta_{k-2l} & = -l(l+1) &\ \ l=0,...,p-2 \\
\delta_{k-2l+1} & = -l(l+1)+1 & \ \ l=1,...,p-2.
\end{array}
$$

As above, $\{n_s,\delta_s\}$ determine the knot Floer homology groups.  
Up to $\Z\oplus\Z$--filtered chain homotopy equivalence, these groups generate $\CFKinf(K)$ as an $\F[U,U^{-1}]$--module.  That is, we have identifications 
\begin{equation}\label{eq:basis} \HFKa_*(K,j)\cong C_*\{0,j\}\cong C_{*-2n}\{-n,j-n\},\end{equation} for all $j,n,*$, where the first isomorphism holds since we work with a reduced representative for $\CFKinf$, and  the second isomorphism is induced by the action of $U^{n}$.  Now a basis for the knot Floer homology groups yields chains $$[\x_s, 0, n_s]   \ \ \ \ \ \ \ \ s=-p+2,...,p-2$$ satisfying $$gr([\x_s,0,n_s])=\delta_s.$$

Equation \eqref{eq:basis} shows that these chains  generate $\CFKinf(T_{p-1,p})$. Since we wish to understand $\CFKinf_0$, it suffices to identify the chains $U^n\cm [\x_s,0,n_s]$
with grading zero.  We have
$$gr(U^n\cm [\x_s,0,n_s])=-2n+\delta_s,$$ from which it follows that  $$U^{-\frac{l(l+1)}{2}} \cm [\x_{k-2l},0,n_{k-2l}] \ \ \ \ \ \ \ \ l=0,...,p-2$$generate $\CFKinf_0$.  But these chains satisfy $$i+j=l(l+1)+n_{k-2l}= g+l^2-lp+2l.$$  Recalling that $g=\frac{(p-2)(p-1)}{2}$, it follows that the sum is bounded below (as $l$ varies) by $\frac{p^2-2p+1}{4}$, as claimed.

The second part of the lemma follows in the same manner.  This time, we find that $U^{-\frac{l(l+1)}{2}}[\x_{k-2l+1},0,n_{k-2l+1}]$ generate $\CFKinf_1$, with filtration values satisfying:
$$i+j=l(l+1)+n_{k-2l+1}= g+l^2-l(p-1)+p-1.$$ 
Here, the minimum value of $\frac{p^2-1}{4}$ occurs when $l=\frac{p-1}{2}$.

As for the last part of the proposition, we claim that every chain in $\CFKinf_{ev}(T_{p-1,p})$ is non-trivial in $\HFinf(S^3)$.  Granting this, the proof is finished: the chain in $\CFKinf_0$ corresponding to $l=\frac{p-3}{2}$ is easily seen to have the desired filtration values. 

To prove the claim, first note that it is enough to prove it for $C_{ev}\{i=0\}$; that is, for the chains in $\CFKinf_{ev}$  identified with the even graded knot Floer homology groups.  This follows from the action of $\F[U,U^{-1}]$ on $\CFKinf$.  Hence it remains to show that each of the  $p-1$ chains above,  $[\x_{k-2l},0,n_{k-2l}]$, represent non-trivial classes in $\HFinf(S^3)\cong \F[U,U^{-1}]$.   

To see this, pick any $[\x_{k-2l},0,n_{k-2l}]$ and consider the diagram of chain complexes and chain maps 
{\small $$\begin{CD} 
 @. @.  \CFKinf(T_{p-1,p})\\
 @. @. @VV\pi V\\
<[\x_{k-2l},0,n_{k-2l}]> @>i_1>> C\{\mathrm{max}(i,j)=n_{k-2l}\} @>i_2>> C\{\mathrm{max}(i,j)\ge n_{k-2l}\} 
\end{CD}$$}

Here, the lower left complex is the complex generated by $[\x_{k-2l},0,n_{k-2l}]$, $i_1$ is its inclusion into the subquotient complex indicated, $i_2$ is the subsequent inclusion into the quotient complex, and $\pi$ is the projection of $\CFKinf$ onto the quotient.

Taking homology, we claim this diagram becomes

$${\small \begin{CD} 
 @. @.  \F[U,U^{-1}]\\
 @. @. @VV\pi_* V\\
\F @>(i_1)_*>> \F @>(i_2)_*>> \F[U,U^{-1}]/U\cm\F[U]
\end{CD} }$$
where $(i_1)_*$ is an isomorphism,  $(i_2)_*$ is an injection, and  $\pi_*$ is surjective.

Write $N = p^2-p+1$.  That $(i_2)_*$ is injective follows from the fact that $i_2$ is chain homotopic, up to an overall grading shift, to the inclusion $$ \CFa(S^3_N(T_{p-1,p})),\spinc_{n_{k-2l}})\hookrightarrow \CFp(S^3_N(T_{p-1,p})),\spinc_{n_{k-2l}}),$$  by Theorem \ref{thmcompute} (since $\CFa$ is the kernel of $U$), together with the fact that $S^3_N(T_{p-1,p})$ is an L-space (and hence the map is isomorphic to the corresponding map $\HFa(S^3)\rightarrow \HFp(S^3)$, up to a grading shift).  

Similar considerations show that $\pi_*$ is surjective; this time Theorem \ref{thmcompute}  shows that  $\pi_*$ is chain homotopic to the map $\HFinf(S^3_N(T_{p-1,p}))\rightarrow \HFp(S^3_N(T_{p-1,p}))$,  which, since $S^3_N(T_{p-1,p})$ is an L-space, is isomorphic to the corresponding map for $S^3$.   

Finally, to see that $(i_1)_*$ is injective (and therefore an isomorphism), we first note that $ [\x_{k-2l},0, n_{k-2l}]$ has strictly larger grading than every other chain in $C\{\mathrm{max}(i,j)=n_{k-2l}\} $, and hence cannot become a boundary under $i_1$.  Similarly,  $\partial \circ i_1\equiv 0$  for grading reasons: If $l>0$, then all other chains have grading at least $3$ less than $\x_{k-2l}$, and hence there are no non-trivial differentials emanating  from $\x_{k-2l}$.  For $l=0$, $\partial i_1(\x_{k})=0$ as well, since  $i_1(\x_k)$ is the only chain in $C\{\mathrm{max}(i,j)=n_{k}\} $ with grading zero, but the homology of this complex is isomorphic to $\HFa(S^3)$ (which is supported in grading zero).

By tracing the the homology class generated by $[\x_{k-2l},0,n_{k-2l}]$ into $$H_*(\CFKinf(T_{p-1,p}))\cong \HFinf(S^3)\cong \F[U,U^{-1}],$$ through the second diagram, we see that it represents a non-trivial class, as claimed.  This completes the proof of the proposition.

 \end{proof}

The structure of $\CFKinf(T_{p-1,p})$ is perhaps best understood through an example, which we include for the reader's convenience.   Figure~\ref{figure45} illustrates a specific subcomplex of $CFK^\infty(T_{4,5})$, shown in the $(i,j)$--filtration plane.  We denote this subcomplex by $C(4,5)$.   Letters represent chains (over $\F=\Z/2\Z$) and an arrow between letters indicates that the terminal chain appears in the boundary of the initial.  The element represented $a$ is at filtration level $(-3,3)$ and has grading $-6$.  The full complex  $CFK^\infty(T_{4,5})$ is generated by $C(4,5)$ as an $\F[U,U^{-1}]$--module; that is, $CFK^\infty(T_{4,5}) = C(4,5) \otimes_\F \F[U,U^{-1}]$.  Thus the full complex has a copy of $C(4,5)$ corresponding to each integer, with each copy specified by the filtration $(n,n)$ of the chain coming from the translate of $d$.  The transformation $U^k$ acts on the total complex by translation by $(-k,-k)$.

 \begin{figure}[h]
 \psfrag{a}{$a$}
 \psfrag{b}{$b$}
 \psfrag{c}{$c$}
 \psfrag{d}{$d$}
 \psfrag{e}{$e$}
 \psfrag{f}{$f$}
 \psfrag{g}{$g$}
\fig{.4}{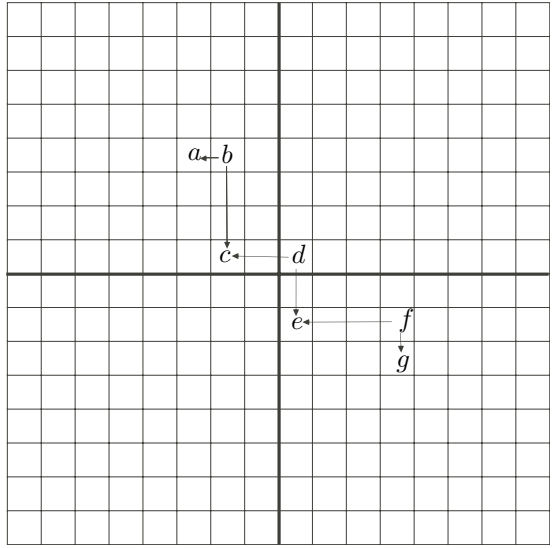}
\caption{A $\Z\oplus\Z$--filtered subcomplex of  $CFK^\infty(T_{4,5})$ that freely generates it as an $\F[U,U^{-1}]$ module.   }\label{figure45}
\end{figure}

The chain complex for the $(p-1,p)$ torus knot is similar.  Again, there is a distinguished subcomplex generating $\CFKinf(T_{p-1,p})$ whose shape resembles a staircase.   Instead of $7$ chains, the general staircase is comprised of $2p-3$ chains.   The first step down  has length $p-2$. Subsequent steps decrease in length by one until arriving at the final step, whose length is one.  The width of the steps follows a similar pattern, beginning with an arrow of width one. As one travels down the staircase, subsequent arrows increase in width by one.  That this subcomplex generates $\CFKinf(T_{p-1,p})$ as an $\F[U,U^{-1}]$--module follows from Theorem $1.2$ of \cite{oz:hfk-lens}.

Returning to the special case of $T_{4,5}$, we point out that the homology of the subcomplex $C(4,5)$ is $\F$, generated by $a$ (or $c$, $e$, or $g$). Thus we have $$H_*(\CFKinf(T_{p-1,p})) \cong \F[U,U^{-1}],$$ as expected ($\CFKinf$ is, after all, a filtered version of $\CFinf(S^3)$).  The grading of $e$  was determined in the course of the  proof of the preceding lemma (where $e$ was called $[x_{-1},0,-2]$).  There, the grading was shown to be  $\delta_{-1}=-6$. Since arrows decrease the grading by one, this determines the grading on the complex completely. 

 An alternative, and often convenient, method for determining the grading follows from the observation that $U^{-3}\cm a$ generates $H_*(C\{i=0\})\cong \HFa_*(S^3)\cong \F.$ Since this latter group is supported in degree zero, and since $U$ carries a grading of $-2$, it follows that $a$ has grading $-6$.

\subsection{Example: computing $d(S^3_{25}(T_{4,5}), \spinc_m)$} \label{subsec:d-example}  To indicate the general route to the correction terms through knot Floer homology, we now compute $d$--invariants for surgery on the $(4,5)$ torus knot.  We restrict our attention to $\SpinC$ structures, $\spinc_m$, for $m=0,5$ and $10$.

According to Theorem~\ref{thmcompute}, in order to find $ d(S^3_{25}(T_{4,5}), \spinc_0)$ we consider the quotient complex  $C\{ \mathrm{max}(i,j)\ge0\}$.  The homology of this quotient is isomorphic to $\F[U,U^{-1}]/U\cm \F[U]$, with $1$  represented by the cycle $a$ (or any cycle homologous to it).  Since $U^{\pm} $ carries a  grading of $\mp 2$, it follows that $a$ represents the element of least grading in $C\{ \mathrm{max}(i,j)\ge0\}$ that is in the image of $U^k$ for all $k\ge 0$.  As a chain in $\CFinf(S^3)$, the discussion above showed that $gr(a)=-6$.  Viewed as a chain in $\CFp(S^3_{25}(T_{4,5}), \spinc_0)$ under the isomorphism of Theorem \ref{thmcompute}, the grading of $a$ must be shifted down by 
$\frac{  -(2m-q)^2 +q}{4q}$, where $m=0$ and $q = 25$. Simplifying, we see that   
$$ d(S^3_{25}(T_{4,5}), \spinc_0) = -6 - \left(\frac{  -(0-25)^2 +25}{4\cm25} \right)= 0.$$

The same approach calculates $ d(S^3_{25}(T_{4,5}), \spinc_5)$. Here we examine the quotient complex $C\{ \mathrm{max}(i,j-5)\ge0\}$.  The only chains in $C(4,5)$ contained in this quotient are $d,e,f,$ and $g$.  By themselves, these chains do not yield non-trivial homology classes, as $e$ is the boundary of $d$ which, in turn, is homologous to $g$.   Similarly, the part of $U^{-1}\cm C(4,5)$ contained in $C\{ \mathrm{max}(i,j-5)\ge0\}$ does not carry non-trivial homology classes.  On the other hand,   $U^{-2}\cm C(4,5)$ is entirely contained within $C\{ \mathrm{max}(i,j-5)\ge0\}$, and hence its  homology contributes to the homology of the quotient.  Indeed, we have $$H_*( C\{ \mathrm{max}(i,j-5)\ge0\})\cong \F[U,U^{-1}]/U^{-1}\F[U], $$ and a representative for the class $U^{-2}$ is provided by the cycle $U^{-2}\cm a$.   As a chain in  $\CFinf(S^3)$ we have $gr(U^{-2}a)=4+gr(a)=4-6=-2$.  Shifting this grading by $s(q,p)$ with $p=5,q=25$ yields  
     
     $$ d(S^3_{25}(T_{4,5}), \spinc_5) = -2 - \left(\frac{  -((2)(5)-25)^2 +25}{4\cm25} \right)= 0.$$

Finally, we perform the same analysis to determine $ d(S^3_{25}(T_{4,5}), \spinc_{10})$.  The first translate of $C(4,5)$ that generates non-trivial homology classes in the relevant quotient  is  $U^{-3}C(4,5)$, 
and we see that $$C\{ \mathrm{max}(i,j-10)\ge0\}\cong \F[U,U^{-1}]/U^{-2}\F[U],$$ with $U^{-3}$ represented by 
the chain $U^{-3}\cm a$.   Applying the shift yields  
$$ d(S^3_{25}(T_{4,5}), \spinc_{10}) =gr(U^{-3}a) - \left(\frac{  -((2)(10)-25)^2 +25}{100} \right)= 6-6+0=0.$$

That all these turn out to be $0$ is expected, since $S^3_{25}(T_{4,5})$ is the $2$--fold branched cover of a smoothly slice knot and hence bounds a rational  homology ball. (The slice knot giving rise to $S^3_{25}(T_{4,5})$ through its branched double cover is the knot obtained by our construction, replacing each Whitehead double in $K_5$ with an unknot.)

\subsection{Dealing with the doubled summand}
As mentioned in the beginning of  this section, we will understand the Floer complex of $L_p = T_{p-1,p} \cs  \frac{3p -1}{2} D(T_{2,3})$ by a K{\"u}nneth-type theorem for connected sums.  To use this, we must understand the key aspects of $CFK^\infty(D(T_{2,3}))$.  We remind the reader that $CFK^\infty$ denotes a particular reduced chain complex representing the $\Z\oplus \Z$--filtered chain homotopy type, and that $D(K)$ is the untwisted, positive-clasped, Whitehead double of $K$.  
We have an analogue of Proposition \ref{gradinglemmaprime} for the Whitehead double.

\begin{prop}  \label{double} The chain complex $\CFKinf_*(D(T_{2,3}))$ satisfies the following:
\begin{itemize}  \item Any chain $[\x,i,j]\in \CFKinf_0(D(T_{2,3}))$ satisfies $i+j\ge 1$.
\item Any chain $[\x,i,j]\in \CFKinf_1(D(T_{2,3}))$ satisfies $i+j\ge 2$.
\item There exist cycles $[\x,0,1],[\y,1,0] \in \CFKinf_0(D(T_{2,3}))$ that are homologous to a generator of $\HFinf_0(S^3)\cong \F$.
\end{itemize}  
\end{prop} 

\begin{proof}
The proof relies on the results of \cite{Doubling}.   Applied to the Whitehead double,  Theorem $1.2$  of \cite{Doubling} shows that

 $$ \HFKa_*(D(T_{2,3}),j)\cong \left\{\begin{array}{ll} 
		\F_{(0)}^{2} \oplus\F_{(-1)}^{2}  & j=1\\
		\F_{(-1)}^{3} \oplus\F_{(-2)}^{4}& j=0\\
		\F_{(-2)}^{2} \oplus\F_{(-3)}^{2} & j=-1\\
	\end{array}\right.$$
The subscripts in the groups refer to the gradings.

Up to $\Z\oplus\Z$--filtered chain homotopy equivalence, a basis forÊ the above groups forms a basis for $\CFKinf(D(T_{2,3}))$Ê as an $\F[U,U^{-1}]$--module; see Equation \eqref{eq:basis}.Ê Henceforth  we let $CFK^\infty(D(T_{2,3}))$ denote this particular representative of the chain homotopy equivalence class.  Since the variable $U$ carries homological degree $-2$Ê and $\Z\oplus\Z$ filtration $(-1,-1)$, we see immediately that any chain $[\x,i,j]\in \CFKinf_{0}(D(T_{2,3}))$ satisfies $i+j\ge 1$.Ê  Indeed, the only chains in $\CFKinf_{0}(D(T_{2,3}))$ are supported in filtration levels $(1,0),(0,1)$, and  $(1,1)$.  This proves the first part of the proposition.  The second part follows similarly: the only chains in $\CFKinf_{1}$ are supported in filtration levels $(1,2),(1,1)$, and  $(2,1)$.



To prove the last part, consider the subcomplex $C\{i<-d\}\subset CFK^\infty$.  Forgetting the additional $\Z$-filtration induced by the knot, we have a chain homotopy equivalence $C\{i<-d\}\simeq U^d
\cm CF^-(S^3)$, where the latter is the subcomplex in the natural filtration of $CF^\infty(S^3)$ described in the first paragraph of Subsection \ref{subsec:HFKbackground}.  The map induced on homology by the inclusion $$\iota: U^d
\cm CF^-(S^3) \rightarrow \CF^\infty(S^3),$$
is easily seen to be injective; indeed, it is given by $$(\iota)_*
: U^{d+1}\cm \F[U] \hookrightarrow \F[U,U^{-1}].$$
Since $1\in \F[U,U^{-1}]$ is the unique class in $HF_{0}^\infty(S^3)$, it follows that the subcomplex $U^{-1}\cm CF^{-}(S^3)\simeq C\{i<1\}\subset CFK^\infty$ contains a cycle homologous to the generator of $HF^\infty_0(S^3)$.    It follows immediately from the description of the knot Floer homology given above, however, that the only chains in $$C\{i<1\}=C\{i\le0\}\subset CFK^\infty(D(T_{2,3}))$$ with homological grading zero have filtration level $(0,1)$.  Thus  there exists a cycle $[\x,0,1] \in \CFKinf_0(D(T_{2,3}))$ which is homologous to a generator of $\HFinf_0(S^3)\cong \F$, as claimed.  

To obtain the other cycle $[\y,1,0]$, it suffices to recall that there is a $\Z\oplus\Z$--filtered chain homotopy equivalence between $\Cdoub$ and the complex obtained from it by interchanging the roles of $i$ and $j$, see~\cite[Section 3.5, Propositions 3.8, 3.9]{Knots}.

\end{proof} 

\subsection{Essential properties of $\CFKinf_*(T_{p-1,p} \cs   \frac{3p -1}{2}  \doub) $}

The following theorem combines Propositions \ref{gradinglemmaprime} and \ref{double} to extract the key features of the chain complex for $L_p$.

\begin{theorem}\label{genthm}

\begin{enumerate}
\item There is an equivalence of $\Z\oplus \Z$--filtered chain complexes $$\CFKinf(T_{p-1,p} \cs   \frac{3p -1}{2}  \doub ) \simeq  \CFKinf(T_{p-1,p}) \otimes_{\Z[U,U^{-1}]} \Cdoub^{\otimes \frac{3p-1}{2}}.$$ 
\item Any chain $[\x,i,j]\in \CFKinf_0(T_{p-1,p} \cs   \frac{3p -1}{2}  \doub )$ satisfies $i + j \ge \frac{p^2+4p-1}{4}$.
\item There exists a cycle $[\x,\frac{p^2-1}{8}, \frac{p^2+8p-1}{8}]\in \CFKinf_0(T_{p-1,p} \cs   \frac{3p -1}{2}  \doub )$  whose homology class generates $\HFinf_{0}(S^3)\cong \F$.
\end{enumerate}
\end{theorem}

\begin{proof} According to  Theorem $7.1$ of \cite{Knots}, the chain complex associated to the connected sums of knots is the filtered tensor product of the complexes associated to the constituent knots.  The first statement is a direct application of this theorem.

The second statement follows from the first, together with propositions \ref{gradinglemmaprime} and \ref{double}.   More precisely, any chain in the filtered tensor product  can be decomposed as a sum $$\underset{l}\Sigma \  \psi_0^l \otimes \theta_1^l\otimes ... \otimes\theta^l_{\frac{3p-1}{2}}$$ where $\psi_0^l\in \CFKinf(T_{p-1,p})$ and $\theta^l_j \in \CFKinf(D(T_{2,3}))$.  Restricting attention to $\CFKinf_0$, we find that each term in the sum satisfies $$0=gr(\psi_0^l \otimes \theta_1^l\otimes ... \otimes\theta^l_{\frac{3p-1}{2}})= gr(\psi_0^l)+gr(\theta_1^l)+...+gr(\theta^l_{\frac{3p-1}{2}}).$$  We may further assume that each of the chains, $\psi_0^l$,  $\theta_j^l$ takes grading values in the set $\{-1,0,1\}$, with the added restriction that the number of chains with grading $+1$ is the same as the number with $-1$. This follows from the $U$ action on the tensor product, since we may write any individual $\theta$ with $|\gr(\theta)| > 1$ as $U^k \theta'$ where $\gr(\theta') \in \{-1,0,1\}$.  Propositions \ref{gradinglemmaprime} and \ref{double} give bounds for the filtration indices associated to any chain with grading $0,1$, and this yields bounds for chains with grading $-1$ as well.  Indeed, any chain $\CFK_{-1}(D(T_{2,3}))$ satisfies $i+j\ge 0$ and  any chain $\CFK_{-1}(T_{p-1,p})$ satisfies $i+j\ge \frac{p^2-9}{4}$, since such chains can be expressed as $U\cm \rho$ for $\rho\in \CFKinf_{1}$.   Additivity of the filtration under tensor product now implies the desired bound.

For the third statement, the desired generator is formed as the tensor product of the cycle $[\x,\frac{p^2-4p+3}{8},\frac{p^2-1}{8}] \in \CFKinf_0(T_{p-1,p})$ from Lemma \ref{gradinglemmaprime},  $p$ copies of the cycle  $[\x,0,1]\in \CFKinf_0(D(T_{2,3}))$ and $\frac{p-1}{2}$ copies of the cycle $[\y,1,0]\in \CFKinf_0(D(T_{2,3}))$ from Proposition \ref{double}. 
\end{proof}

\subsection{Computing $d$}

 The work of the previous subsections was aimed at the following non-vanishing theorem for the $\bar{d}$--invariant.

\begin{theorem}\label{t:dbarcomp}  With the notation above, $d(S^3_{p^2}(L_p), \spinc_0)  \le  -p-1 $ and  $d(S^3_{p^2}(L_p), \spinc_p) = -p+1$.  In particular, $\bar{d}(S^3_{p^2}(L_p), \spinc_p) \ge 2$.

\end{theorem}

\begin{proof}  By Theorem \ref{genthm}, all chains in $\CFKinf_0(L_p)$ have $(i,j)$--filtration indices satisfying $i + j \ge \frac{p^2+4p-1}{4}$.  This holds, in particular, for any cycle $\theta$ homologous to a generator of $\HFinf_0(S^3)$.   For each such cycle,   it follows that $U^\frac{p^2+4p+3}{8}\theta$ has filtration level satisfying $i + j \ge -1$.  Thus $U^\frac{p^2+4p+3}{8}\theta$ has one of $i$ or $j$ nonnegative, so that $U^\frac{p^2+4p+3}{8}\theta \in C\{ \mathrm{max}(i,j)\ge0\}$ and can be viewed as homology class in $\HFp(S^3_{p^2}(L_p),\spinc_0)$ under the identification given by Theorem \ref{thmcompute}.  By equivariance of $\CFKinf$ with respect to the action of $\F[U,U^{-1}]$, we see that $U^\frac{p^2+4p+3}{8}\theta$ is a cycle whose  homology class generates  the summand of $\HFinf(S^3)$ in grading   $$gr(U^\frac{p^2+4p+3}{8}\theta)=-2\cm \frac{p^2+4p+3}{8}+gr(\theta)=-\frac{(p+1)(p+3)}{4}+0.$$  Such a cycle is in the image of $U^i$ for all $i$, so that if we view it as a (non-zero) homology class in $\HFp(S^3_{p^2}(L_p),\spinc_0)$, its grading gives an upper bound for $d(S^3_{p^2}(L_p), \spinc_0) $.  This grading, in turn, is determined by the degree shift formula to be $-\frac{ (p+1)(p+3)}{4} - \frac{-p^2 +1}{4}= -p -1$.  (Here we use $m = 0$ and $q = p^2$ to compute the degree shift, $\frac{-p^2 +1}{4}$).

For the second statement,  consider the generator $\rho=U^{\frac{p^2-1}{8}}\cm [\x,\frac{p^2-1}{8}, \frac{p^2+8p-1}{8}]$, where $[\x,\frac{p^2-1}{8}, \frac{p^2+8p-1}{8}]$ is given by the third statement in Theorem~\ref{genthm}.  Observe that  $$\rho= [\x,0,p]\in C\{ \mathrm{max}(i,j-p)\ge0\}$$ whereas 
$$U^k\cm\rho = [\x,-k,p-k] \notin C\{ \mathrm{max}(i,j-p)\ge0\},$$
for all $k>0$.  By equivariance, $\rho$  is in the image of $U^i$ for all $i\ge0$.  Moreover, the first observation above, together with Theorem \ref{thmcompute},  allows us to view $\rho$ as a homology class in $\HFp(S^3_{p^2}(L_p),\spinc_p)$.

  Now this class may be zero.  Indeed, while $[\rho]$ is non-zero in $\HFinf(S^3)$ it could be a boundary in the quotient complex $C\{ \mathrm{max}(i,j-p)\ge0\}$.   The algebraic mechanism by which this could occur is illustrated in the example of Section  \ref{subsec:d-example}.  There, the cycle $e$ generated $\HFinf_{-6}(S^3)$, but was null-homologous in the quotient complex used for the computation of $d(S^3_{25}(T_{4,5}),\spinc_5)$.  
  
  On the other hand, since $\CFKinf$ is finitely generated as an $\F$--module for any fixed degree, we know that for all  $*\gg0$,  $$ C_{*}\{ \mathrm{max}(i,j-p)\ge0\}\cong \CFinf_{*}(S^3),$$ Hence $[U^{-k}\rho]$ is a non-trivial homology class in either group for sufficiently large $k>0$. It follows that 
   $$\mathrm{min}\{ gr(U^{-k}\rho)\ |\   [U^{-k}\rho]\ne0 \in \HFp(S^3_{p^2}(L_p),\spinc_p)\}$$   is well-defined and equals $d(S^3_{p^2}(L_p),\spinc_p)$.   The second observation above shows that $[U^k\cm \rho] =0 \in  \HFp(S^3_{p^2}(L_p),\spinc_p)$ for all $k>0$, so that  the grading of $\rho$, shifted by the quantity $s(p^2,p)$, provides a lower bound for $d(S^3_{p^2}(L_p),\spinc_p)$.  Explicitly, we have $$gr(\rho)-s(p^2,p)=-2\cm \frac{p^2-1}{8}-\frac{-p^2 +4p -3}{4}=-p+1,$$ and we see that $d(S^3_{p^2}(L_p),\spinc_p)\ge-p+1$, as claimed.

\end{proof}

Given the result of Theorem~\ref{t:dbarcomp}  that $\bar{d}(S^3_{p^2}(L_p), \spinc_p) \ge 2$, we have now completed   the final details for the proof of Theorem~\ref{infgen} as outlined in the introduction.


\section{Rational homology cobordisms}\label{cobordismsection}

The study of knot concordance is closely related to the study of    homology cobordism of rational homology spheres.  To make this formal, denote by $\Omega_{spin}^{\Q}$ the group of smooth spin $\Q$--homology $3$--spheres, modulo   smooth spin $\Q$--homology cobordisms.   For any prime-power $p^k$, there is a homomorphism $\calc  \to \Omega_{spin}^\Q$ induced by taking $p^k$--fold branched covers (see~\cite[\S 2]{grs} for a discussion of the spin structure) and many invariants of knot concordance (such as the ones used in this paper) factor through these maps.  From this point of view,  {\em integral} homology spheres are analogous to knots with Alexander polynomial $1$.

As an alternative, the  $2$--fold branched cover of a knot $K$ is a $\Z/2\Z$--homology sphere.  If we denote by $\Omega^{\Z/2\Z}$ the group of $\Z/2\Z$--homology spheres modulo $\Z/2\Z$--homology cobordism, the corresponding homomorphism maps $\calc \to \Omega^{\Z/2\Z}$. 

For the remainder of this section, we use $\Omega$ to denote either of $\Omega_{spin}^{\Q}$ or $\Omega^{\Z/2\Z}$.

In either case, $\Omega$ contains a subgroup generated by integral homology spheres.  We denote this subgroup $\Omega_I$.  Recall that Freedman's simply-connected surgery theory implies that any element in $\Omega$ represented by an {\it integral} homology sphere  bounds a topological homology ball, and thus $\Omega_I$ is in the kernel of the homomorphism $\Omega \to \Omega^{top}$, where $\Omega^{top}$ denotes the corresponding topological cobordism group.  Denote this kernel by $\Omega_T$.   Thus $\Omega_I \subset \Omega_T$.

By analogy with the question of whether a topologically slice knot is smoothly concordant to a knot of polynomial $1$  (that is, does $\calc_\Delta = \calc_T$?), we may ask whether $\Omega_I = \Omega_T$.   Theorem~\ref{infgen} provides a negative answer, which we may state in the following terms. 

\begin{theorem}\label{infgenQ} 
  $\Omega_{T}/\Omega_{I}$ contains an infinitely generated free subgroup, where $\Omega = \Omega^{\Q}_{spin}$ or  
  $\Omega = \Omega^{\Z/2\Z}$.
\end{theorem}
\begin{proof}  The $2$--fold branched cover of any knot,  $\Sigma(K)$, is a $\Z/2\Z$--homology ball, and thus also  a rational homology ball. Since $K_p$ is topologically slice, the $2$--fold branched cover of $B^4$ over the slice disk is a  $\Z/2\Z$--homology ball, and thus also a rational homology ball.  This cover, and its boundary $\Sigma(K_p)$, have unique spin  structures.  Hence, for all $p$,  $\Sigma(K_p)$ represents a class in $\Omega_T$.  
   The proof of Theorem~\ref{infgen} shows that no linear combination of the $\Sigma(K_p)$ together with an integral homology sphere is trivial in $\Omega$.  Thus, $\{ \Sigma(K_p)\}_{p\in P}$ form an infinite linearly independent set in $\Omega_T / \Omega_I$ (here $P$ is, as before, the set of primes congruent to $3$ modulo $4$). 
\end{proof}

The rational homology balls bounded by the double branched covers $\Sigma(K_p)$ are interesting from the point of view of $4$--dimensional handlebody theory.  It follows from Cerf theory that a compact $4$--manifold $M$ that has a handlebody structure is smoothable; a weaker structure on $M$ is discussed in~\cite[Problem 4.74]{kirby:problems96} and is sometimes called a {\em pseudo-handlebody structure}.  This is a decomposition of $M =  M_0 \cup_\Sigma \Delta$ where $M_0$ is smooth (and hence has a handle decomposition relative to its boundary) and $\Delta$ is contractible.  Examples of manifolds without such structures were constructed by Stong and Taylor; see, for example, the discussion after~\cite[Problem 4.74]{kirby:problems96}.   Theorem~\ref{infgenQ} gives rise to further examples of such manifolds.
\begin{prop}
Suppose that $W$ is a topological rational homology ball and that $\partial W$ is not smoothly rationally homology cobordant to an integral homology sphere.  Then $W$ does not have a pseudo-handlebody structure.
\end{prop}
\begin{corollary}
Let $W_p$ be the $2$--fold cover of $B^4$, branched along the topological slice disk for $K_p$.  Then $W_p$ does not have a pseudo-handlebody structure.
\end{corollary}

\section{Boundary links}\label{sectionboundary}
Our techniques also apply to understanding an important issue that arises in the theory of link concordance.  Recall that a link $L = L_1 \cup \cdots \cup L_k$ is called a {\em boundary link} if its components bound disjoint Seifert surfaces.   It is well-known that not every link is concordant to a boundary link. The methods for demonstrating this~\cite{cochran-orr:boundary-links, livingston:boundary-links, gl,   mi1, mi2}   typically show that a given link is not topologically (locally flat) concordant to a boundary link.  As in knot concordance, this raises the question addressed in this section: {\em is every link that is topologically concordant to a boundary link in fact smoothly concordant to a boundary link?}  The $2$--component link $L_p$ illustrated in Figure~\ref{f:blink}, where $-p$ denotes half twists between the bands, provides a counterexample.

\begin{theorem}\label{t:blink}
For any $p \equiv 3 \mod 4$, the $2$--component link in Figure~\ref{f:blink} is topologically slice (and hence concordant to a boundary link) but not smoothly concordant to a boundary link.
\end{theorem}
\begin{figure}[h]
\psfrag{jp}{$J_p$}
\psfrag{kjp}{$K_p$}
\psfrag{alpha}{$\alpha$}
\psfrag{mpo2}{$-p$}
\psfrag{band}{band move}
\fig{.6}{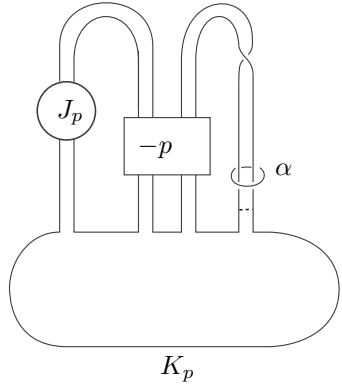}
\caption{The link $L_p = K_p \cup \alpha$.}\label{f:blink}
\end{figure}

The proof requires some preparatory material.

\begin{lemma}\label{l:linkingnumber} 
 Suppose that a link $L = (L_1, L_2) \subset S^3 \times \{1\}$ is concordant to a boundary link $L' = (L_1', L_2') \subset S^3 \times \{0\}$  via a concordance $C \cong C_1 \amalg C_2 \subset S^3 \times [0,1]$.  Let $\Sigma_i'$ be disjoint Seifert surfaces for $L_1',\ L_2'$, and let $\Sigma _i = C_i \cup_{L_i'} \Sigma_i'$.  Finally, view $B^4$ as the union along $S^3 \times \{0\}$ of $S^3 \times [0,1]$ with a $4$--ball, and view all of the surfaces just described as embedded in $B^4$.
Then the map $H_1(\Sigma_1) \to H_1(B^4 - \Sigma_2)$ induced by inclusion is trivial. 
\end{lemma}
This follows from the fact that curves on $\Sigma_1'\subset S^3 \times \{0\}$ have trivial linking number with $L_2'$. 
\begin{proof}[{\bf Proof of Theorem~\ref{t:blink}}] For simplicity the link will simply be referred to as $L = K \cup \alpha$.  To see that $L$ is a topologically slice link, do the band move indicated by the dotted line in Figure~\ref{f:blink}.  After the band move indicated in Figure~\ref{f:blink}, the link becomes two parallel copies of $J$, separated from $\alpha$ by a $2$--sphere.  The band gives rise to a genus  $0$ cobordism in $S^3 \times [0,1]$ from $L$ to two parallel copies of $J$ having linking number $0$.  Moreover, this cobordism  is separated from $\alpha \times [0,1]$.  Glue a $4$--ball to $S^3 \times [0,1]$ as above, and recall that $J$ is chosen to be topologically slice, so that we can cap off both copies of $J$ with slice disks, and $\alpha\times \{0\}$ with a disk as well, to give a topological slice for $L$.

Suppose that $L$ is smoothly concordant to a boundary link, by a concordance $C = C_K \cup C_\alpha$.  Let $r$ be a reflection of $S^3$, and let $-C$ be the image of $C$ under $r \times id_I: S^3 \times I$.  Gluing $C$ to $-C$ along the component $C_K$ would then give a concordance from the $3$--component link $\hat L_p$ drawn in Figure~\ref{f:Lhat} to a boundary link. 
\begin{figure}[h]
\psfrag{jp}{$J_p$}
\psfrag{kjp}{$K_p$}
\psfrag{alpha}{$\alpha$}
\psfrag{beta}{$\beta$}
\psfrag{mjp}{$-J_p$}
\psfrag{mkjp}{$-K_p$}
\psfrag{po2}{$p$}
\psfrag{mpo2}{$-p$}
\fig{.6}{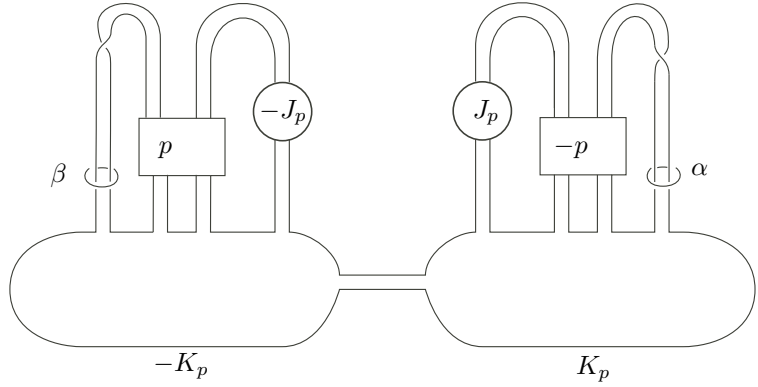}
\caption{$\hat L_p = K_p \# -K_p \cup \alpha \cup \beta$.}\label{f:Lhat}
\end{figure}

To complete the proof of the theorem, we show that $\hat L_p$ is not concordant to a boundary link.    If $\hat L_p$ were concordant to a boundary link, then the component of the boundary link corresponding to (the slice knot) $K_p \# -K_p$ would be slice.  Attach  a $4$--ball to the $S^3 \times I$ containing the concordance, and add on a slicing disk to that component, yielding a slicing disk $D$ for $K_p \# -K_p$ in $B^4$.  The knots $\alpha$ and $\beta$ bound disjoint surfaces in the complement of $D$, and according to Lemma~\ref{l:linkingnumber} the inclusions of these surfaces in $B^4 - D$ induce trivial maps in $H_1$.  This is illustrated schematically in Figure~\ref{f:config}.
\begin{figure}[h]
\psfrag{Khat}{$K_p \#-K_p$}
\psfrag{alpha}{$\alpha$}
\psfrag{s3i}{$S^3 \times I$}
\psfrag{b4}{$B^4$}
\psfrag{beta}{$\beta$}
\psfrag{Delta}{$D$}
\fig{.6}{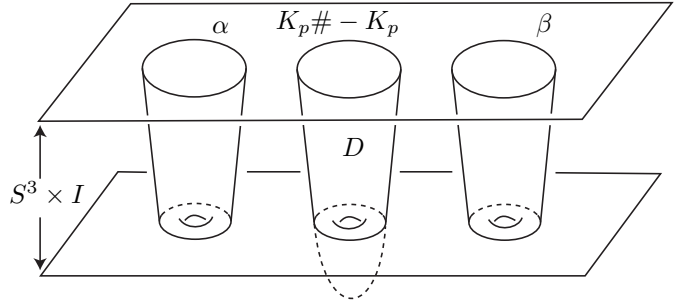}
\caption{Concordance to a boundary link, with one component filled in by a slicing disk.}\label{f:config}
\end{figure}

Let $W$  be the $2$--fold branched cover of  $B^4$, branched along $D$; this is a rational homology ball and its boundary is $\Sigma(K_p \# -K_p)$, which is diffeomorphic to the connected sum $\Sigma(K_p) \# -\Sigma(K_p)$.  Since $H_1(\Sigma(K_p \# -K_p)) \cong \Z_{p^2} \oplus \Z_{p^2}$, we know that $\ker\left[H_1(\Sigma(K_p \# -K_p) )\to H_1(W)\right]$ has order $p^2$.   To determine this kernel more precisely, observe that by Lemma~\ref{l:linkingnumber},  $H_1$ of the surfaces bounded by $\alpha$ and $\beta$ becomes trivial in $B^4 - D$, so those surfaces lift to $W$.  In particular, the inverse image of $\alpha$ consists of two curves, each of which is null-homologous in $W$, and similarly for $\beta$.  Let us write $\tilde\alpha$ for one of those lifts, and $\tilde\beta$ for one of the lifts of $\beta$.

A surgery picture for $\Sigma(K_p \# -K_p)$ is obtained from a copy of the diagram in Figure~\ref{figure2} for $\Sigma(K_p)$, together with its reflection.  The curve $\tilde\alpha$ appears as a meridian of the $-1$ framed unknot; when that unknot is blown down to produce Figure~\ref{figure3}, $\tilde\alpha$ twists $p$ times around the knot $T_{p-1,p} \cs J_p \cs J_p^r$.  Similarly, $\tilde\beta$ twists $p$ times around $-(T_{p-1,p} \cs J_p \cs J_p^r)$.   Since $H_1(\Sigma(K_p \# -K_p))$ is generated by the meridians of those two knots, it follows that $\tilde\alpha$ and $\tilde\beta$ generate $\ker\left[H_1(\Sigma(K_p \# -K_p)) \to H_1(W)\right] \cong \Z_{p} \oplus \Z_{p}$.

Writing, as before $\spinc_i \oplus \spinc_j$ for the $\SpinC$  structure corresponding to the element $(i,j) \in H_1(\Sigma(K_p \# -K_p))$, we see that $d((\Sigma(K_p) \# -\Sigma(K_p),\spinc_{pi} \oplus \spinc_{pj}) = 0$ for all $i,j$.  But 
$$
d((\Sigma(K_p) \# -\Sigma(K_p),\spinc_{pi} \oplus \spinc_{pj}) = d(\Sigma(K_p),\spinc_{pi})  - d(\Sigma(K_p), \spinc_{pj}).
$$
Taking $j=0$, we see that $\bar{d}$ vanishes on the subgroup $p\Z/{p^2}\Z$.  This contradicts Theorem~\ref{t:dbarcomp}, and hence the theorem is proved.
\end{proof}

\appendix
\section{Metabolizers for $(\Z/ p^2 \Z)^n$ and $\bar{d}$}

 Let $M \subset H_1(nY_p) \cong (\Z / p^2\Z)^n$ be a subgroup of order $p^n$ on which $ \bar{d}$ vanishes.  We let $M_p = \{m\in M \ |\ pm = 0\}$.  To each    $m = p(m_1, \ldots , m_n) \in M_p$ we have a relation $\sum \bar{d}(Y_p, \spinc_{pm_i}) = 0$. Our goal is to show that these relations are sufficient to conclude $\bar{d}(Y_p, \spinc_{kp})  = 0$ for all $k$.
 We state this as a theorem, the proof of which occupies this appendix.
 
 \begin{theorem}  Let $ M \subset H_1(nY_p) \cong (\Z / p^2\Z)^n$ be a metabolizer, and suppose that $\bar{d}(nY_p, \spinc_m) = 0$ for all $m \in  M$.  Then $\bar{d}(Y_p, \spinc_{ip}) = 0$ for all $i$.
 
 \end{theorem}

  \subsection{Special elements in the metabolizer}

We begin by  showing that any metabolizer contains element of  a special type.
 
  \begin{theorem}\label{specvect} If $M \subset   (\Z/ p^2 \Z)^n $ has order $p^n$, then it contains an element of the form $z = (b_1, b_2, \ldots, b_n)$ where at least $n/2$ of the $b_i$ are equal to $p \in  \Z/ p^2 \Z$ and all $b_i$ are multiples of $p$.
 
 \end{theorem}
 
 \begin{proof} Any generating set for $M$ must have at least $n/2$ elements.  (This~holds whether $n$ is even or odd.)     Let a minimal generating set for $M$ consist of elements $v_i$, $1 \le i \le N$. After perhaps rearranging the order  of the summands of  $(\Z/ p^2 \Z)^n$, a change of generating sets (corresponding to performing row operations in the Gauss-Jordan algorithm) yields generators of the form $$w_i = (0, \ldots, 0, w_{i,i}, 0 , \ldots, 0, w_{i,N}, \ldots , w_{i,n}), \hskip.1in 1 \le i \le N,$$ where $w_{i,i} \ne 0$ and  if $w_{i,i} $ is a nonzero multiple of $p$, then so are all $w_{i,j}$.
 
 Multiplying each $w_i$ by an appropriate element of $\Z/ p^2\Z$ we get a set of elements   $z_i \in M \subset (\Z/ p^2 \Z)^n$ of the form $z_i = (0, \ldots, 0, p, 0 , \ldots, 0, w_{i,N}, \ldots , w_{i,n})$, $1 \le i \le N$, where now each $w_{i,j}$ is divisible by $p$.  (If any of the multipliers is a multiple of $p$, this set will no longer generate $M$.)  Finally, let $z$ be the sum of the $z_i$, giving us the element $z = (p, p, \ldots, p, b_{N+1}, \ldots , b_{n})$ where each entry of $z$ is a multiple of $p$.
 \end{proof}

 \subsection{The relations space for $\bar{d}(Y_p,\spinc_{pm_i})$}
  
  To simplify notation we write $\bar{d_i} = \bar{d}(Y_p,\spinc_{p i})$ for $1 \le i \le \frac{p -1}{2}$.  Recall that $p \equiv 3 \mod 4$, and thus $p = 2q +1$ for some odd $q$. It follows that $ \frac{p -1}{2} = q$ is an integer.  Let $\calr = \{ (\alpha_1, \cdots , \alpha_q) \in \Q^q\ | \  \sum \alpha_i \bar{d}_i = 0 \}$.  Note that unless all $\bar{d}_i = 0$, $\calr$ is a $(q-1)$--dimensional subspace of $\Q^q$.
  
  Each element in $M_p$ determines an element in $\calr$.  We denote this map (not a homomorphism) by $\psi:$
  
  $$\psi (  p(m_1, \ldots , m_n))= (\alpha_1, \ldots, \alpha_q),$$ where $\alpha_j$ is the number of $m_i = \pm j \mod q$.
  
  If we view the integers $\{1, \ldots , q\}$ as representatives of the   multiplicative group $\Z_p^*/\{\pm 1\}$, (a cyclic group of order $q$) then a multiplicative generator of $\Z_p^*$, say $a$, gives a cyclic permutation of $  \{1, \ldots , q\}$ of order $q$.    For example, if $p = 23$ then $5$ generates $\Z_{23}^*$.  Multiplying the elements of  $ L =  \{1, \ldots , 11\}$ by $5$ and,  if need be multiplying by $-1$ to arrive at elements in $L$, gives a map of order $11$.
  
 $$(1,2,3,4,5,6,7,8,9,10,11) \to (5,10, 15, 20, 25, 30, 35, 40, 45, 50, 55)  \equiv $$  $$(5,10, -8, -3, 2, 7, -11, -6, -1, 4, 9)$$  The resulting permutation in cyclic notation is:
 
 $$(1, 5, 2,10, 4, 3, 8, 6,7, 11, 9).$$  This defines a cyclic action $\rho$ on $\Q^q$.  
 
 Multiplication by $a$ also acts on $M_p$ and the map $\psi$ is equivariant with respect to $\rho$.  That is, the subspace of $\calr$ generated by $\psi(M_p)$ is invariant under the action of $\rho$.  As we now see, there are very few invariant subspaces.
 The action $\rho$ makes $\Q^q$ into
  a free $\Q[\Z_q]$--module:  $\Q[\Z_q] = \Q[t]/\left< t^q - 1 \right>$. 
   This module splits into a direct sum of cyclic summands:  $$ \Q[t]/\left< t^q - 1 \right> \cong  \Q[t]/\left< t - 1 \right> \oplus_d \Q[t]/ \left<\phi_d(t)\right>,$$  where the $\phi_d$ are the $d$--cyclotomic polynomials and $d$ ranges over all nontrivial divisors of $q$.   
    
 \subsection{Conclusion of proof}
  The vector $z$ given by Theorem~\ref{specvect} 
  satisfies $\psi(z) = (n, b_2, \ldots , b_q)$ where
   $\sum b_j \le n$ and all $b_j \ge 0$.
     Viewed as an element in $\Q[\Z_q]$, 
      $\psi(z) = f_z(t) =  n + \beta_1 t + \beta_2t^2 + \cdots + \beta_{q-1}t^{q-1}$ where
    the $\beta_i$ are some permutation of the $b_i$.
     If $\psi(z)$ were in some proper invariant subspace of $\Q[\Z_q]$, then $f_z(t)h(t)$ would be a multiple of $t^{q}-1$ for some proper divisor $h(t)$ of $t^q -1$.  This would imply that $f_z(\omega) =0$ for some $q$--root of unity.   But any such (odd) root of unity has real part greater than $-1$; considering the real part of the $f_z(\omega)$ and the fact that the constant term is at least as large as the sum of the remaining coefficients, this is impossible.
     
     In conclusion, $\calr = \Q^q$ and in particular every relation $\bar{d}_i = 0$ is among the relations.
     
     \subsection{Example}  To clarify the previous discussion, we present an example.  Suppose that we are in the case $p=31$ and $n= 8$.  Theorem~\ref{specvect} tells us that within the metabolizer for $H_1(Y_p) \cong (\Z_{31^2})^8$ there is a vector with all entries divisible by $31$ and at least four of them equal to $31$, for instance $$(31, 31, 31, 31, 13(31), 13(31), 27(31), 0).$$
 The presence of this vector in the metabolize would imply that:$$4\bar{d}_1 + 2\bar{d}_{13} + \bar{d}_{27} = 0.$$
 Since $\bar{d}_{27} = \bar{d}_4$, we rewrite this as $$4\bar{d}_1+  \bar{d}_{4} + 2\bar{d}_{13}  = 0.$$
 
 The group $\Z_{31}^*/\{\pm 1\}$ is cyclic generated by $3$.  The first 15 powers of $3$ are:
 $$\{1, 3, 9, 27, 19, 26, 16, 17, 20, 29, 25, 13, 8, 24, 10\}.$$ Replacing $x$ with $-x \mod 31$ when $x \ge 15$ we have $$\{1, 3, 9, 4, 12, 5, 15, 14, 11, 2, 6, 13, 8, 7, 10\}.$$
 Multiplication by $3$ permutes $ \Z_{31}^*/\{\pm 1\}$, acting as a 15--cycle.  If we identify $\Q[\Z_{15}]$ with $\Q[t]/\left< t^{15} -1\right>$, identifying $t^i$ with $\bar{d}_{3^i}$ we have the relation $$4 + 2 t^3+ t^{11} = 0.$$  The set of all relations between the $\bar{d}_i$ corresponds to an ideal in $\Q[\Z_{15}]$, but the ideal generated $4 + 2t^3 +t^{11}$ is all of $\Q[\Z_{15}]$, since $4 + 2t^3 +t^{11}$ is relatively prime to $t^{15}-1$.  In particular, $t^i$ is in the ideal, so $\bar{d}_{3^i} = 0$ for all $i$.
 
  

 \vfill \eject

\newcommand{\etalchar}[1]{$^{#1}$}

\end{document}